\documentclass[a4paper,11pt,DIV=11,
twocolumn,
abstract=on
]{scrartcl}

\usepackage{algorithm, booktabs}
\usepackage{mathtools}
\mathtoolsset{showonlyrefs}
\usepackage{amsmath}
\usepackage{amssymb}
\usepackage{array}
\usepackage{amsthm}
\usepackage{subcaption}
\usepackage{mathdots}
\usepackage{mathrsfs}
\usepackage{algorithm}
\usepackage[noend]{algpseudocode}
\usepackage{graphicx}
\usepackage{color}
\usepackage{listings, enumitem}
\usepackage{tikz}
\usepackage{hyperref}
\usepackage{float}

\normalfont
\newcommand{\N}{\ensuremath{\mathbb{N}}}
\newcommand{\M}{\mathcal{M}}

\newcommand{\eps}{\varepsilon}

\newcommand{\R}{\ensuremath{\mathbb{R}}}
\newcommand{\C}{L^\infty}

\newcommand{\mj}{\setminus j}
\newcommand{\X}{\mathbb{X}}
\newcommand{\Y}{\mathbb{Y}}

\newcommand{\supp}{\textnormal{supp}}

\newcommand{\norm}[1]{\left\Vert #1\right\Vert}

\newcommand{\dx}{\,\mathrm{d}}
\newcommand{\tT}{\mathrm{T}}

\newcommand{\aprox}{\textnormal{aprox}}

\newcommand{\dual}{\mathcal{D}_\varepsilon^\varphi} 

\newcommand\curfolder{img/}
\newcommand\curwidth{0.24\textwidth}

\def\3{\ss}

\newmuskip\pFqmuskip

\newcommand*\pFq[6][8]{
  \begingroup 
  \pFqmuskip=#1mu\relax
  \begingroup\lccode`\~=`\,
  \lowercase{\endgroup\let~}\pFqcomma
  {}_{#2}F_{#3}{\left(\genfrac..{0pt}{}{#4}{#5};#6\right)}%
  \endgroup
}
\newcommand*\pRegFq[6][8]{
  \begingroup 
  \pFqmuskip=#1mu\relax
  \begingroup\lccode`\~=`\,
  \lowercase{\endgroup\let~}\pFqcomma
  {}_{#2}\tilde{F}_{#3}{\left(\genfrac..{0pt}{}{#4}{#5};#6\right)}%
  \endgroup
}

\newcommand{\pFqcomma}{\mskip\pFqmuskip}
\newcommand*{\bigtimes}{\mathop{\raisebox{-.5ex}{\hbox{\Large{$\times$}}}}}

\DeclareMathOperator*{\diag}{diag}

\DeclareMathOperator{\F}{F}

\DeclareMathOperator{\G}{\mathcal{G}}
\DeclareMathOperator{\V}{\mathcal{V}}
\DeclareMathOperator{\E}{\mathcal{E}}

\DeclareMathOperator*{\argmin}{argmin}

\DeclareMathOperator{\KL}{KL}
\DeclareMathOperator{\TV}{TV}
\DeclareMathOperator{\UOT}{UOT}
\DeclareMathOperator{\UMOT}{UMOT}
\DeclareMathOperator{\MOT}{MOT}

\DeclareMathOperator{\dom}{dom}

\newtheorem{theorem}{Theorem}[section]
\newtheorem{lemma}[theorem]{Lemma}
\newtheorem{remark}[theorem]{Remark}
\newtheorem{definition}[theorem]{Definition}
\newtheorem{example}[theorem]{Example}
\newtheorem{corollary}[theorem]{Corollary}
\newtheorem{proposition}[theorem]{Proposition}

\begin{document}
\title{Unbalanced\\
 Multi-Marginal Optimal Transport}

\author{
Florian Beier\footnotemark[1]
	\and
Johannes von Lindheim\footnotemark[1]
	\and
Sebastian Neumayer\footnotemark[1]
	\and
Gabriele Steidl\footnotemark[1]
	}

\maketitle

\footnotetext[1]{Institute of Mathematics,
	Technische Universit\"at Berlin,
	Strasse des 17. Juni 136, 10587 Berlin, Germany,
	\{f.beier,vonlindheim\}@tu-berlin.de, \{neumayer,steidl\}@math.tu-berlin.de}
	
\begin{abstract}
Entropy regularized optimal transport and its multi-marginal generalization have 
attracted increasing attention in various applications, in particular due to efficient Sinkhorn-like algorithms for computing optimal transport plans. 
However, it is often desirable that the marginals of the optimal transport plan do not match the given measures exactly, which led to the introduction of the so-called unbalanced optimal transport.
Since unbalanced methods were not examined for the multi-marginal setting so far, we address this topic in the present paper.
More precisely, we introduce the unbalanced multi-marginal optimal transport problem and its dual, and show that a unique optimal transport plan exists under mild assumptions.
Furthermore, we generalize the Sinkhorn algorithm for regularized unbalanced optimal transport to the multi-marginal setting and prove its convergence.
For cost functions decoupling according to a tree, the iterates can be computed efficiently. 
At the end, we discuss three applications of our framework, namely two barycenter problems and a transfer operator approach, where we establish a relation between the barycenter problem and the multi-marginal optimal transport with an appropriate tree-structured cost function.
\end{abstract}

\noindent\textbf{Mathematics Subject Classification.} 49Q22, 49Q20, 49M29, 65D18, 37M10.

\noindent\textbf{Keywords.} Entropy regularization, multi-marginal optimal transport, Sinkhorn algorithm, unbalanced optimal transport, Wasserstein barycenters.

\section{Introduction} \label{sec:intro}

Over the last decades, optimal transport (OT) has attracted increasing attention in various applications, e.g., barycenter problems \cite{AC2011,BCCN15}, image matching \cite{WSBO13,ZYHT07} and machine learning \cite{ACB17,FZM15,KSKW15}.
As the OT minimization problem is numerically difficult to solve, a common approach is to add an entropy regularization term \cite{C2013}.
This enables us to approximately solve the problem using Sinkhorn iterations \cite{Sin64} by exploiting an explicit relation between the solutions to the corresponding primal and dual problems.
These iterations can be implemented in parallel on GPUs, which makes even large scale problems solvable within reasonable time.
Recently, a debiased version of the Sinkhorn divergence was investigated in \cite{FSVATP2018,NS20,RTC17}, which has the advantage that it characterizes the weak convergence of measures.
However, in many applications, the assumption that the marginal measures are matched exactly appears to be unreasonable.
To deal with this issue, unbalanced optimal transport (UOT) \cite{CPSV17, CPSV18,LMS18} was introduced. 
Here, the hard marginal constraints are replaced by penalizing the $\varphi$-divergences between the given measures and the corresponding marginals. 
By making minimal modifications, we can use Sinkhorn-like iterations and hence the computational complexity and scalability remain the same.
For the unbalanced case, the mentioned debiasing was discussed in \cite{SFVP19}.
For Gaussian distributions, the corresponding divergence even has a closed form \cite{JMPC20}.

So far we commented on OT between two measures.
For certain practical tasks such as matching for teams \cite{CE10}, 
particle tracking \cite{CK18}, and information fusion \cite{DC14,EHJK20}, 
it is useful  to compute transport plans between more than two marginal measures. 
This is done in the framework of multi-marginal optimal transport (MOT) \cite{Pass15}, 
where again entropy regularization is possible.
The problem was tackled numerically for Coulomb cost in \cite{BCN16}, 
and more general repulsive costs in \cite{CDD15,GKR20}.
Later, an efficient solution for tree-structured costs 
using Sinkhorn iterations was established in \cite{HRCK20}.
More recently, these results were extended to even more general cost functions in \cite{AB20}.

In this paper, we want to combine UOT and MOT by 
investigating unbalanced multi-marginal optimal transport (UMOT).
We can build upon the previous papers, but we will see that this has to be done carefully, since various generalizations that seem to be straightforward at the first glance appear to be not.
Inspired by the papers \cite{MG2019} and \cite{SFVP19}, we formulate Sinkhorn iterations for the UMOT problem. Our approach differs from that in \cite{SFVP19} as we cannot rely on the $1$-Lipschitz continuity of the $(c,\varepsilon)$-transform, which only holds in the two-marginal case.
Instead, we  establish uniform boundedness of the iterates and then exploit the compactness of the Sinkhorn operator.
As in the two-marginal case, we retain the excellent scalability of the algorithm for tree-structured costs.
Furthermore, we prove that these iterations are convergent under mild assumptions.

As one possible application, we discuss the computation of regularized UOT barycenters 
based on the regularized UMOT problem.
The OT barycenter problem was first introduced in \cite{AC2011} and then further studied, 
e.g., in \cite{BCCN15,cohen2021sliced,DDGU18,LSPC19} for the balanced setting.
As soon as entropic regularization is applied, we usually obtain a blurred barycenter.
One possible solution is to use the debiased Sinkhorn divergence instead \cite{JCG20, quang2020entropic}.
Similarly as in \cite{HRCK20} for the balanced case, we observe that solving an UMOT problem instead of minimizing a sum of UOT ``distance'' reduces the blur considerably.
We also validate this observation theoretically.
For this purpose, we show that for tree-structured costs UMOT transport plans are already determined by their two-dimensional marginals.
A complementary approach for computing barycenters in an unbalanced setting is based on the Hellinger--Kantorovich distance \cite{CP20,FMS21}.

Furthermore, we provide a numerical UMOT example with a path-structured cost function.
We observe that in comparison with a sequence of UOT problems with the same input measures, the UMOT approach has several advantages since the target measure from some optimal UOT plan is not necessarily equal to the source measure of the subsequent UOT plan.
For the (single) UMOT plan, this is impossible by construction.
Additionally, coupling the problems allows information to be shared between them.
Note that there is almost no computational overhead compared to solving the problems sequentially.

\textbf{Outline of the paper:} 
Section~\ref{sec:basics} contains the necessary preliminaries.
The regularized UMOT problem, in particular the existence and uniqueness of a solution as well as its dual problem are provided in Section~\ref{cha:mm}.
In Section~\ref{sec:sink}, we derive a Sinkhorn algorithm for solving the regularized UMOT problem and prove its convergence.
Then, in Section~\ref{sec:BaryInt}, we investigate the barycenter problem with respect to regularized UOT and establish a relation to the regularized UMOT problem with a tree-structured cost function, where the tree is simply star-shaped. Furthermore, we discuss an extension of these considerations to more general tree-structured cost functions.
Additionally, we outline an efficient implementation of the Sinkhorn algorithm for tree-structured costs.
Numerical examples, which illustrate our theoretical findings, are provided in Section~\ref{sec:numerical_examples}.
Finally, we draw conclusions in Section~\ref{sec:conc}.

\section{Preliminaries} \label{sec:basics}
Throughout this paper, let
$\X$ be a compact Polish space with associated Borel $\sigma$-algebra $\mathcal B(\X)$.
By $\M(\X)$, we denote the space of finite signed real-valued Borel measures, which can be identified via Riesz' representation theorem with the dual space of the continuous functions $C(\X)$ endowed with the norm $\| f\|_{C(\X)} \coloneqq \max_{x \in \X} |f(x)|$.
Denoting the associated dual pairing by $\langle \mu , f \rangle \coloneqq \int_\X f \, d\mu$,
the space $\M(\X)$ can be equipped with the weak*-topology, 
i.e., a sequence $(\mu_n)_{n \in \N} \subset \M(\X)$ converges \emph{weakly} to $\mu \in \M(\X)$, 
written $\mu_n \rightharpoonup \mu$, if
    \[\int_\X f \, d\mu_n \to \int_\X f \, d\mu \quad \text{for all } f \in C(\X).\]
The associated dual norm of $\mu \in \M(\X)$, also known as \emph{total variation}, is given by $\mathrm{TV} (\mu) = \sup_{\|f\|_{C(\X)}  \le 1} \langle \mu, f \rangle$.		
By $\mathcal M^+(\X)$, we denote
the subset of non-negative measures.
The \emph{support} of $\mu \in \M^+(\X)$ is defined as the closed set 
\begin{align*}
\supp(\mu) \coloneqq \bigl\{ x \in \X: &\,B \subset \X \text{ open, }x \in B \\ &\implies \mu(B) >0\bigr\}.
\end{align*}
For $\mu \in \mathcal M^+(\X)$ and $p \in [1,\infty]$, let $L^p(\X,\mu)$ 
be the Banach space (of equivalence classes) of real-valued measurable functions with norm $\|f\|_{L^p(\X,\mu)} < \infty.$

A measure $\nu \in \mathcal M(\X)$ is called \emph{absolutely continuous} with respect to $\mu$, and we write $\nu \ll \mu$, if for every $A \in \mathcal B(\X)$ with $\mu(A) = 0$ we have $\nu(A) = 0$.
For any $\mu, \nu \in \mathcal M^+(\X)$ with $\nu \ll \mu$, the \emph{Radon--Nikodym derivative} 
$$\sigma_\nu \coloneqq \frac{\dx \nu}{\dx \mu} \in L^1(\X,\mu)$$ 
exists and $\nu = \sigma_\nu \mu$.
Furthermore, $\mu, \nu \in \mathcal M(\X)$ are called \emph{mutually singular}, 
denoted by $\mu \perp \nu$, if there exist two disjoint sets $X_\mu, X_\nu \in \mathcal B(\X)$ such that $\X = X_\mu \dot \cup X_\nu$ and for every $A \in \mathcal B(\X)$ we have $\mu(A) = \mu(A \cap X_\mu)$ and $\nu(A) = \nu(A \cap X_\nu)$.
Given $\mu,\nu \in \mathcal M^+(\X)$, there exists a unique \emph{Lebesgue decomposition} $\mu = \sigma_\mu \nu + \mu^\perp$, 
where $\sigma_\mu \in L^1(\X,\nu)$ and $\mu^\perp \perp \nu$.
Let $\Y$ be another compact Polish space and $T\colon \X \to \Y$ be a measurable function, i.e., $T^{-1}(A) \in \mathcal{B}(\X)$ for all $A \in \mathcal{B}(\Y)$.
Then, the \emph{push-forward measure of $\mu$ by $T$} is defined as $T_\# \mu  \coloneqq \mu \circ T^{-1}$.

Let $V$ be a real Banach space with dual $V^*$ and dual pairing $\langle v,x \rangle = v(x)$, $v \in V^*, x \in V$.
For $F\colon V \rightarrow (-\infty,+\infty]$, the \emph{domain} of $F$ is given by
$ \mathrm{dom} F \coloneqq \{x \in V: F(x) \in \mathbb R \}$.
If $\mathrm{dom} F \not = \emptyset$, then $F$ is called \emph{proper}.
By $\Gamma_0(V)$ we denote the set of proper, convex, lower semi-continuous (lsc) functions mapping from $V$ to $(-\infty,+\infty]$.
The \emph{subdifferential} of $F\colon V \rightarrow (-\infty,+\infty]$ at a point $x_0 \in \mathrm{dom} F$
is defined as
\begin{align*}
&\partial F(x_0)\\
\coloneqq &\bigl\{q \in V^*: F(x) \ge F(x_0) + \langle q,x - x_0 \rangle \bigr\},
\end{align*}
and $\partial F(x_0) = \emptyset$ if $x_0 \not \in \mathrm{dom} F$. 
For a function $F\colon V^n \to (-\infty,\infty]$, $\partial_i F$ denotes the subdifferential of $F$ with respect to the $i$-th component.
The \emph{Fenchel conjugate}
$F^*\colon V^* \rightarrow (-\infty,+\infty]$ is given by
$$
F^*(q) = \sup_{x \in V} \{ \langle q,x\rangle  - F(x) \}.
$$

A non-negative function $\varphi \in \Gamma_0(\R)$ satisfying $\varphi(1) = 0$ and $\varphi|_{(-\infty,0)} = \infty$ 
is called \textit{entropy function} with \textit{recession constant} 
$\varphi^\prime_\infty = \lim_{x \to \infty} \varphi(x)/x$.
In this case, $\dom(\varphi^*) = (-\infty, \varphi^\prime_\infty]$.
For every $\mu,\nu \in \mathcal M^+(\X)$ with Lebesgue decomposition $\mu = \sigma_\mu \nu + \mu^\perp$, 
the \emph{$\varphi$-divergence} $D_\varphi \colon \mathcal M^+(\X) \times \mathcal M^+(\X) \to [0,\infty]$ is given 
in its primal and dual form by 
\begin{align}
D_\varphi(\mu,\nu) 
&= \int_{\X} \varphi \circ \sigma_\mu \dx \nu + \varphi^\prime_\infty\, \mu^\perp(\X)\\
&=
\sup_{f \in C(\X)} \langle \mu, f \rangle - \langle \nu, \varphi^* \circ f \rangle \label{eq:dual_div}
\end{align}
with the convention $0 \cdot \infty = 0$.
The mapping
$D_\varphi$ is jointly convex, weakly lsc and fulfills $D_\varphi(\mu,\nu) \ge 0$, see \cite[Cor.~2.9]{LMS18}.
Furthermore, we have for $t > 0$ that $D_{t\varphi} = t D_{\varphi}$.
We will use the following $\varphi$-divergences, see also~\cite{SFVP19}.

\begin{example}\label{ex:1}\hspace{.01cm}
\begin{itemize}
\item[i)]
Let $\varphi \coloneqq \iota_{\{1\}}$, where $\iota_{\mathcal S}$ denotes the \emph{indicator function} 
of the set ${\mathcal S}$,
i.e., $\iota_{\mathcal S}(x) = 0$ if $x \in {\mathcal S}$, and $\iota_{\mathcal S}(x) = +\infty$ otherwise.
Then $\varphi^*(q) = q$, $\varphi'_\infty = \infty$, and
\begin{equation} \label{eq:Dphi1}
D_{\varphi}(\mu,\nu)   = \left\{
\begin{array}{ll}
   0 &\mathrm{if} \; \mu=\nu,\\
	+\infty &\mathrm{otherwise}.
	\end{array}
	\right.
\end{equation}         
\item[ii)] For $\varphi\coloneqq \iota_{[0,\infty)}$, we get
$\varphi^*(q) = \iota_{(-\infty,0]}$, $\varphi'_\infty = 0$, and
$D_{\varphi}(\mu,\nu)   = 0$.
\item[iii)] Consider the \emph{Shannon entropy} with $\varphi(x) \coloneqq x\ln(x) - x +1$ and the agreement $0 \ln 0 = 0$.
Then, we have that $\varphi^*(q) = \exp(q) -1$, $\varphi'_\infty = \infty$, and
the $\varphi$-divergence is the
\emph{Kullback--Leibler divergence} 
$\mathrm{KL}\colon{\mathcal M^+}(\X) \times {\mathcal M^+}(\X) \rightarrow \mathbb [0, +\infty]$.
For $\mu,\nu\in {\mathcal M^+}(\X)$, if the Radon--Nikodym derivative 
$\sigma_\mu = \frac{\dx \mu}{\dx \nu}$
exists, then
\begin{equation} \label{KLdef}
\mathrm{KL} (\mu,\nu) \!\coloneqq \!\int_{\X} \ln(\sigma_\mu ) \, \dx \mu + \nu(\X) - \mu(\X),
\end{equation}
and otherwise, we set $\mathrm{KL} (\mu,\nu) \coloneqq + \infty$.
Note that the $\mathrm{KL}$ divergence is strictly convex with respect to the first variable. 
\item[iv)] For $\varphi(x) \coloneqq |x-1|$, we get
$\varphi^*(q) = \max(-1,q)$ if $ q \le 1$ and $\varphi^*(q) = +\infty$ otherwise, $\varphi'_\infty = 1$, and
$
D_{\varphi}(\mu,\nu)   = \mathrm{TV}(\mu - \nu)
$.
\end{itemize}
\end{example}

\section{Unbalanced Multi-Marginal Optimal Transport \label{cha:mm}}
Throughout this paper, we use the following abbreviations.
For compact Polish spaces $\X_i \neq \emptyset$, $i=1,\ldots, N$, and measures $\nu_i\in \M^+(\X_i)$, $i=1,\ldots, N$,
we set $\nu \coloneqq (\nu_1,\ldots,\nu_N)$ and
\begin{align}
    &\X \coloneqq \bigtimes_{i=1}^N \X_i, \qquad \X_{\mj} \coloneqq \bigtimes_{\scriptsize\substack{i=1 \\ i\neq j}}^N \X_i,\\  
    &\nu^\otimes \coloneqq \bigotimes_{i=1}^N \nu_i, \qquad 
    \nu_{\mj}^\otimes \coloneqq \bigotimes_{\scriptsize\substack{i=1 \\ i\neq j}}^N \nu_i.
\end{align}
Furthermore, for $p \in [1,\infty]$ and $f_i \in L^p(\X_i, \nu_i)$, $i=1,\ldots,N$, we write
$f \coloneqq  (f_1,\ldots,f_N)$ and
\begin{align*}
L^{p,\times}(\X,\nu) \coloneqq \bigtimes_{i=1}^N L^p(\X_i,\nu_i),\\
f^\oplus \coloneqq \bigoplus_{i=1}^N f_i, \quad
f_{\mj}^\oplus \coloneqq \bigoplus_{\scriptsize\substack{i=1 \\ i\neq j}}^N f_i,
\end{align*}
where the product space $L^{p,\times}(\X,\nu)$ is equipped with the $L^p$ norm of the components.
For example, if $\X=\X_1\times \X_2\times \X_3$, then for $f=(f_1, f_2, f_3)\in L^{p,\times}(\X,\nu)$ we have
\begin{align*}
f^\oplus_{\setminus 2}(x_1, x_2, x_3) = f_1(x_1) + f_3(x_3).
\end{align*}
If the domains and measures are clear from the context, we abbreviate the associated norms by $\Vert \cdot \Vert_p$, $p \in [1,\infty]$.

Given $\mu_i\in \M^+(\X_i)$, $i=1,\ldots, N$,
the measures $\gamma_i\in \M^+(\X_i)$, $i=1,\ldots, N$, are called \emph{reference measures} for $\mu$ if
\[
\KL(\mu^\otimes, \gamma^\otimes)< \infty.
\]
\begin{definition}
    Let $\varepsilon>0$. Given a non-negative cost $c \in C (\X)$, fully supported measures $\mu_i \in \mathcal{M}^+(\X_i)$, $i=1,\ldots,N$, with reference measures $\gamma_i \in \mathcal{M}^+(\X_i)$, and entropy functions $\varphi_i \in \Gamma_0(\R)$, $i=1,\ldots,N$, the associated 
    \emph{regularized unbalanced multi-marginal optimal transport problem $(\UMOT)$} reads
    \begin{align}\label{eq:reg_mm_ub_OT}
        \UMOT_\eps(\mu) \coloneqq &\!\!\inf_{\pi \in \M^+(\X)} \int_{\X}c\dx \pi
        + \eps \KL\bigl(\pi, \gamma^\otimes \bigr) \nonumber \\
        &+ \sum_{i=1}^N D_{\varphi_i}(P_{\X_i} \strut{}_\# \pi, \mu_i),
\end{align}
where $P_{\X_i} \strut{}_\# \pi$ is the $i$-th marginal of $\pi$.
\end{definition}
Note that the full support condition is no real restriction as we can choose $\X_i = \supp(\mu_i)$.
Furthermore, we can implicitly incorporate weights for the marginal penalty terms in \eqref{eq:reg_mm_ub_OT} by rescaling the entropy functions $\varphi_i$.

\begin{remark}[Regularization and reference measures]\label{rem:kl}
Typical choices for the reference measures $\gamma_i$ are
\begin{itemize}
\item[i)] $\gamma_i = \mu_i$ then $\KL(\mu^\otimes, \gamma^\otimes) = 0$ and we regularize in $\UMOT_\eps$
by $\KL(\pi, \mu^\otimes)$. 
\item[ii)] $\gamma_i$ the Lebesgue measure on  $\X_i \subset \R^{d_i}$.
Then $\KL(\mu^\otimes, \gamma^\otimes) < \infty$  is equivalent to $\mu_i$ 
having a density in a so-called Orlicz space, see \cite{CLMW19,NS20} for details.
Furthermore, the regularizer in $\UMOT_\eps$ is the entropy $\KL(\pi, \gamma^\otimes) = E (\pi)$ for continuous measures.
\item[iii)] $\gamma_i$ the counting measure if the $\X_i$ are finite.
Here, $\KL(\mu^\otimes, \gamma^\otimes) < \infty$ is equivalent to $\mu$ being positive.
Then, the regularizer is the entropy for discrete measures 
$\KL(\pi, \lambda^\otimes)= E(\pi)$.
\end{itemize}
\end{remark}
Definition \eqref{eq:reg_mm_ub_OT} includes the following special cases:
\begin{itemize}
\item If $\varphi_i = \iota_{\{1\}}$ for all $i=1,\ldots,N$, then we have by Example \ref{ex:1} i)
the regularized multi-marginal optimal transport ($\mathrm{MOT}_{\eps}$)
with hard constraints for the marginals. For $\eps =0$, we deal with the plain multi-marginal optimal transport (MOT) formulation.
\item If $N=2$, then we are concerned with regularized unbalanced optimal transport ($\mathrm{UOT}_{\eps}$).
If $\varphi_1 = \varphi_2 = \iota_{\{1\}}$, we get regularized optimal transport ($\mathrm{OT}_{\eps}$), and if $\eps = 0$,
we arrive at the usual optimal transport (OT) formulation.
\end{itemize}

Regarding existence and uniqueness of minimizers for $\UMOT_\eps$, we have the following proposition.
\begin{proposition}\label{prop:ub_existence}
The $\UMOT_\eps$ problem \eqref{eq:reg_mm_ub_OT} admits a unique optimal plan.
\end{proposition}
\begin{proof}
The problem is feasible due to $\KL(\mu^\otimes, \gamma^\otimes) < \infty$.
Existence follows since $\KL(\cdot,\gamma^\otimes)$ and hence the whole functional \eqref{eq:reg_mm_ub_OT} is coercive, and since all involved terms are lsc.
For the uniqueness note that all terms in \eqref{eq:reg_mm_ub_OT} are convex in $\pi$ and that KL is moreover strictly convex in its first argument.
\end{proof}

For applications, the dual formulation of $\UMOT_\eps$ is important.

\begin{proposition}\label{prop:duality_reg_mm_ub_OT}
The $\UMOT_\eps$ problem         \eqref{eq:reg_mm_ub_OT} admits the dual representation
    \begin{align}
  \!\!\!\!\!\!\! \UMOT_\eps&(\mu) =  \!\!\!\!\sup_{f\in L^{\infty,\times}(\X,\gamma)} \!\!\!
  \dual(f) + \eps \gamma^\otimes(\X), \label{eq:dual_problem_reg_mm_ub_OT}
        \end{align}
    where
    \begin{align}
    \dual(f) \coloneqq
    &- \sum_{i=1}^N \int_{\X_i}  \varphi_i^*\circ(-f_i) \dx \mu_i  \nonumber\\
    &-  \eps \int_{\X} \exp\Bigl(\frac{ f^\oplus - c}{\eps}\Bigr)\dx \gamma^\otimes.
  \label{eq:dual_problem_reg_mm_ub_OT_functional}
    \end{align}
    
    The optimal plan $\hat \pi \in \M^+(\X)$ for the primal problem \eqref{eq:reg_mm_ub_OT} 
		is related to any tuple of optimal dual potentials $\hat f \in L^{\infty,\times}(\X,\gamma)$ by
  \begin{equation} \label{eq:pd} 
		\hat \pi= \exp\Bigl(\frac{\hat f^\oplus - c}{\eps}\Bigr)\gamma^\otimes.
	\end{equation}
Furthermore, any pair of optimal dual potentials $\hat f, \hat g \in L^{\infty,\times}(\X,\gamma)$ satisfies $\hat f^\oplus = \hat g^\oplus$.
Moreover, if $N-1$ of the $\varphi_i^*$, $i =1,\dotsc,N$, are strictly convex, then it holds $\hat{f} = \hat{g}$.
\end{proposition}

\begin{proof}
  First, we set $V \coloneqq L^{\infty,\times}(\X,\gamma)$ and $W \coloneqq \C(\X, \gamma^\otimes)$,
    and define $A\colon V \to W$, $F\in \Gamma_0(V)$ and $G\in \Gamma_0(W)$ via
        \begin{align*}
            &A(f) \coloneqq f^\oplus,\\
            &F(f) \coloneqq \sum_{i=1}^N \int_{\X_i}  \varphi_i^*\circ f_i \dx \mu_i,\\
            &G(f) \coloneqq \eps \int_{\X} \exp\Bigl(\frac{f  - c}{\eps}\Bigr) - 1 \dx \gamma^\otimes.
        \end{align*}
    Note that $V^*$ and $W^*$ are the respective dual spaces of finitely additive signed measures that are absolutely continuous with respect to $\gamma$ and $\gamma^\otimes$, respectively.
    From \cite[Thm.~4]{Rock68} with the choice $g\colon \X\times\R\to \R$,
    \begin{align*}
    g(x, p) &\coloneqq \eps \Big( \exp \Big( \frac{p-c(x)}{\eps} \Big)-1 \Big)
    \end{align*}
    such that $g^*(x, q)=c(x)q + \eps (\ln(q)q-q+1)$ if $q\geq 0$ with the convention $0\ln 0=0$ and $g^*(x, q)=\infty$ otherwise,
    we get that the Fenchel conjugate $G^* \in \Gamma_0(W^*)$
		is given by
        \[G^*(\pi) = \int_{\X} c \sigma_\pi \dx \gamma^\otimes + \eps \KL\bigl(\sigma_\pi \gamma^\otimes, \gamma^\otimes\bigr)\]
        if there exists non-negative $\sigma_\pi \in L^1(\X,\gamma^\otimes)$ with $\langle \pi, f \rangle = \int_\X f \sigma_\pi \dx \gamma^\otimes$ for all $f \in W$ and  $G^*(\pi) = \infty$ for all other $\pi\in W^*$.
		Using the definition of the Fenchel conjugate and \eqref{eq:dual_div}, the function $F^* \circ A^*\in \Gamma_0(W^*)$ can be expressed for any such $\pi$ as
        \begin{align*}
            &F^*(A^*\pi)\\
            = &\!\!\!\!\sup_{f\in L^{\infty,\times}(\X,\gamma)}\!\langle A^* \pi, f \rangle - \!\sum_{i=1}^N \int_{\X_i}  \varphi_i^*\circ f_i \dx \mu_i\\
            = &\!\!\!\!\sup_{f\in L^{\infty,\times}(\X,\gamma)} \langle \pi, f^\oplus \rangle - \sum_{i=1}^N \int_{\X_i}  \varphi_i^*\circ f_i \dx \mu_i\\
            = &\!\!\!\!\sup_{f\in L^{\infty,\times}(\X,\gamma)} \sum_{i=1}^N \int_{\X_i} \!f_i \dx P_{\X_i} \strut{}_\# \pi  - \!\!\int_{\X_i}  \!\varphi_i^*\circ f_i \dx \mu_i\\
            = &\sum_{i=1}^N D_{\varphi_i}(P_{\X_i} \strut{}_\# \pi, \mu_i).
        \end{align*}
Now, we obtain the assertion by applying the Fenchel--Rockafellar duality  relation 
\begin{align} \label{primal-dual}
&\inf_{w \in W^*} \bigl\{ F^*(A^* w) + G^*(w)\bigr\}\\
=
&\sup_{x \in V}\bigl\{ - F(-x) - G(Ax)\bigr\},\notag
\end{align}
see \cite[Thm.~4.1, p.~61]{ET1999}.
Due to the definition of $G^*$, it suffices to consider elements from $W^*$ that can be identified with elements in $L^1(\X,\gamma^\otimes)$.
Hence, the problem \eqref{primal-dual} coincides with \eqref{eq:reg_mm_ub_OT}.

The second assertion follows using the optimality conditions. More precisely, let $\hat{\pi}$ and $\hat{f}$ be optimal. By \cite[Chap.~3,~Prop.~4.1]{ET1999}, this yields $A \hat f \in \partial G^* (\hat \pi)$ which is equivalent to $\hat{\pi} \in \partial G(A\hat{f})$.
Since $G$ is G\^ateaux-differentiable with $\nabla G(f) = \exp \bigl(\frac{f - c}{\eps}\bigr) \gamma^\otimes$, we obtain
\[
\hat{\pi} = \exp \biggl(\frac{\hat{f}^\oplus - c}{\eps}\biggr) \gamma^\otimes.
\]

Finally, let $\hat{f},\hat{g} \in L^{\infty,\times}(\X,\gamma)$ be two optimal dual potentials.
The second summand is concave in $f$ and strictly concave in $\hat f^\oplus$.
The first summand in \eqref{eq:dual_problem_reg_mm_ub_OT_functional} is concave.
Moreover, if $N-1$ of the $\varphi_i$, $i=1,\dotsc,N$, are strictly convex, then \eqref{eq:dual_problem_reg_mm_ub_OT_functional} is strictly concave.
Hence, both claims follow.
\end{proof}

\section{Sinkhorn Algorithm for Solving the Dual Problem} \label{sec:sink}
In this section, we  derive an algorithm for solving 
the dual problem \eqref{eq:dual_problem_reg_mm_ub_OT}. 
We prove its convergence under the assumption that for all $i = 1, \ldots, N$, we have
$\ln(\sigma_{\mu_i}) \in L^\infty(\X_i,\gamma_i)$,
where $\sigma_{\mu_i}$ is the Radon-Nikodym derivative of $\mu_i$ with respect to the reference measures $\gamma_i$, and some mild assumptions on the entropy functions $\varphi_i$.

First, we introduce two operators that appear in the optimality conditions of the dual problem, 
namely the $(c,\eps)$-transform and the anisotropic proximity operator.
\begin{definition}
For $j = 1,\dotsc,N$, the \emph{$j$-th $(c,\eps)$-transform} $\F^{(c,\eps,j)}\colon L^{\infty,\times}(\X,\gamma)\to L^\infty(\X_j,\gamma_j)$ is given by
\begin{align} \label{Fcej}
    &\F^{(c,\eps,j)}(f) = f^{(c,\eps,j)}
    \\
    &\coloneqq  \eps \ln (\sigma_{\mu_j}) - \eps \ln \biggl(\int_{\X_{\mj}} \!\!\!\exp\biggl(\frac{f_{\mj}^\oplus - c}{\eps}\biggr) \dx \gamma_{\mj}^\otimes\biggr).\notag
\end{align}
\end{definition}
This transform was discussed in relation with $\MOT_\varepsilon$ in \cite{MG2019}, where the following two properties were shown.
\begin{lemma}\label{prop:properties_entropic-c-trafo}
    Let $\ln(\sigma_{\mu_i}) \in L^\infty(\X_i,\gamma_i)$, $i = 1,\dotsc,N$, and $f \in L^{\infty,\times}(\X,\gamma)$.
    Then, the following holds:
	\begin{itemize}
	    \item[i)] For every $j = 1,\dotsc,N$ it holds 
			\begin{align*}
			&\Vert f^{(c,\eps,j)} \!+ \lambda_{f,j} \Vert_{\infty} 
			\!\leq\! \Vert c \Vert_\infty 
			\!+ \eps \Vert \ln (\sigma_{\mu_j}) \Vert_{\infty},
			\end{align*}
			where
	    \[
        \lambda_{f,j} \coloneqq \eps \ln \biggl(\int_{\X_{\mj}} \exp\biggl(\frac{f_{\mj}^\oplus}{\eps}\biggr) \dx \gamma_{\mj}^\otimes\biggr).
        \]
    In particular, we get that $f^{(c,\eps,j)}$ has bounded oscillation
	\[
        \sup_{y \in \X_j} f^{(c,\eps,j)}(y)
        - \inf_{y \in \X_j} f^{(c,\eps,j)}(y) < \infty.
    \]		
    \item[ii)] The nonlinear and continuous operator $\F^{(c,\eps,j)}\colon L^{\infty,\times}(\X,\gamma)\to L^p(\X_j,\gamma_j)$ is compact for $p \in [1, \infty)$, i.e., it maps bounded sets to relatively compact sets.
    	\end{itemize}
\end{lemma}

\begin{definition}
For any entropy function $\varphi \in \Gamma_0(\R)$ and $\eps > 0$, 
the \emph{anisotropic proximity operator} $\aprox_{\varphi^*}^{\eps}\colon \R \to \R$ is given by
\begin{equation}\label{eq:aprox_op_def}
    \aprox_{\varphi^*}^{\eps}(p) \coloneqq \argmin_{q \in \R} \bigl\{\eps e^{\frac{p-q}{\eps}} + \varphi^*(q)\bigr\}.
\end{equation}
\end{definition}
\begin{remark}
This operator is indeed well-defined.
Furthermore, it is $1$-Lipschitz, and can be given in analytic form for various conjugate entropy functions, see \cite{SFVP19}.
\end{remark}

\begin{example}\label{ex:2}
Let us have a closer look at the functions from Example~\ref{ex:1}.
\begin{itemize}
\item[i)]
For $\varphi = \iota_{\{1\}}$ it holds $\aprox_{\varphi^*}^{\eps}(p) = p$.
\item[ii)] For $\varphi=\iota_{[0,\infty)}$, we get 
$\aprox_{\varphi^*}^{\eps}(p) = 0$.
\item[iii)] For $\varphi(x) = t (x \ln(x) -x+1)$ corresponding to the Kullback--Leibler divergence, we have
\begin{equation} \label{aniso_kl}
 \aprox_{\varphi^*}^{\eps}(p) = \frac{t}{t+\eps} \, p.
\end{equation}
\item[iv)] For  $\varphi(x) = t |x-1|$ belonging to the $\TV$ distance, it holds
\begin{equation} \label{aniso_tv}
 \aprox_{\varphi^*}^{\eps}(p) = 
\left\{
\begin{array}{rl}
-t,& p < -t,\\
p, & p \in [-t,t],\\
t, & p > t.
\end{array}
\right.
\end{equation}
\end{itemize}
\end{example}

\begin{definition}
The $(c,\eps)$-transform and the anisotropic proximity operator
are concatenated to the \emph{$j$-th Sinkhorn mapping}
\[
S^{(c,\eps,\varphi,j)} \colon L^{\infty,\times}(\X,\gamma) \to \C(\X_j,\gamma_j)
\]
defined as $S^{(c,\eps,\varphi,j)}(f) \coloneqq f^{(c,\eps,\varphi,j)}$ with
\[f^{(c,\eps,\varphi,j)} \coloneqq - \aprox_{\varphi_j^*}^\eps \bigl( -f^{(c,\eps,j)} \bigr),\]
where the operator \smash{$\aprox_{\varphi_j^*}^\eps$} is applied pointwise.
\end{definition}

Now, we derive the maximizer of $\dual$ defined in \eqref{eq:dual_problem_reg_mm_ub_OT_functional}.

\begin{proposition}\label{prop:sinkhorn_estimate}
    Let $\ln(\sigma_{\mu_i}) \in L^\infty(\X_i,\gamma_i)$, $i=1,\ldots,N$.
 Then, it holds for all $j = 1,\dotsc,N$ and $f \in L^{\infty,\times}(\X,\gamma)$ that
    \begin{align*}
        &\dual(f) \leq{} \dual(\dotsc,f_{j-1},f^{(c,\eps,\varphi,j)},f_{j+1},\dotsc)
    \end{align*}
    with equality if and only if $f_j = f^{(c,\eps,\varphi,j)}$.
	    Furthermore, $f$ is a maximizer of $\dual$ if and only if 
		$f_j = f^{(c,\eps,\varphi,j)}$ for all $j=1,\ldots,N$.
\end{proposition}

\begin{algorithm*}[t]
	\begin{algorithmic}
	    \State \textbf{Input:} $f^{(0)} \in L^{\infty,\times}(\X,\gamma)$ with $\dual(f^{(0)}) > - \infty$
			\State \textbf{Iterations:}
		\For{$r = 0,1,\ldots$ }
		   \For{$j = 1,2,\ldots,N$}
			    \State 
					$f_i^{(rN+j)} =
					\left\{         
					\begin{array}{ll}
					\left( f^{(rN+j-1)} \right)^{(c,\eps,\varphi,j)} & i=j\\
					f_i^{(rN+j-1)} &i\not =j
					\end{array}
					\right.$
				\EndFor	
			\EndFor
			\caption{Sinkhorn Iterations for $\mathrm{UMOT}_\eps$}
		\label{def:sinkhorn}
	\end{algorithmic}
\end{algorithm*}

\begin{proof}
We fix $j \in \{1,\dotsc,N\}$ and rewrite 
\begin{align*}
    \dual(f)
    = - \!\sum_{i \neq j} \bigl \langle \mu_i, \varphi_i^*(-f_i) \bigr \rangle - \!\! \int_{\X_j}\!\! \varphi_j^*(-f_j) \sigma_{\mu_j}\\
    + \eps \exp\Bigl(\frac{f_j}{\eps}\Bigr) \int_{\X_{\mj}} \!\!\exp\biggl(\frac{f_{\mj}^\oplus - c}{\eps}\biggr) \dx \gamma_{\mj}^\otimes \dx \gamma_j.
\end{align*}
By rearranging the definition \eqref{Fcej} of the $(c,\eps)$-transform, we get that
\begin{align*}
    &\exp\biggl(\frac{f^{(c,\eps,j)}}{-\eps} \biggr) = \frac{1}{\sigma_{\mu_j}} \int_{\X_{\mj}} \!\!\!\exp\biggl(\frac{f_{\mj}^\oplus - c}{\eps}\biggr) \dx \gamma_{\mj}^\otimes.
\end{align*}
Plugging this into the equation above, we get
\begin{align*}
    \dual(f) = &- \sum_{i \neq j} \langle \mu_i, \varphi_i^*(-f_i) \rangle - \int_{\X_j} \varphi_j^*(-f_j)\\ &+ \eps \exp\Bigl(\frac{f_j - f^{(c,\eps,j)}}{\eps}\Bigr) \dx \mu_j.
\end{align*}
The integrand on the right-hand side has the form of the functional in \eqref{eq:aprox_op_def}, so that
we obtain
\begin{align*}
    &\dual(f) \leq \dual(\dotsc,f_{j-1},f^{(c,\eps,\varphi,j)},f_{j+1},\dotsc).
\end{align*}
Since the minimization problem \eqref{eq:aprox_op_def} admits a unique solution,
strict inequality holds if and only if $f_j \neq f^{(c,\eps,\varphi,j)}$.

By the first part of the proof the relation
$f_j = f^{(c,\eps,\varphi,j)}$ is equivalent to 
$0 \in \partial_j \dual(f)$, and the last statement follows if we show that
\[\partial \dual(f) = \bigtimes_{i=1}^N \partial_i \dual(f).\]
Using $F_i\colon \C(\X_i, \gamma_i) \to [-\infty,\infty)$ as well as $G\colon L^{\infty,\times}(\X,\gamma) \to \R$ given by
\begin{align*}
    &F_i(f) = -\int_{\X_i} \varphi_i^* \circ (-f) \dx \mu_i, \quad i = 1,\dotsc,N,\\
    &G(f) = -\eps \int_{\X} \exp\Bigl(\frac{f  - c}{\eps}\Bigr) \dx \gamma^\otimes,
\end{align*}
respectively, we can decompose $\dual$ as $\dual = F^\oplus + G$.
As $G$ is G\^ateaux-differentiable, we obtain that $\partial G = \bigtimes_{j=1}^N \partial_j G$.
By continuity of $G$ and since 
$0\in \dom \varphi^*$,
it holds $0\in \dom G \cap \dom F^\oplus$,
such that the subdifferentials are additive by \cite[Ch.~1, Prop.~5.6]{ET1999}.
Thus, using $\partial F^\oplus = \bigtimes_{i=1}^N \partial F_i$, we obtain
\begin{align*}
    \partial \dual  &= 
		\partial G + \partial F^\oplus = 
		\bigtimes_{i=1}^N \partial_i G + \bigtimes_{i=1}^N \partial_i F^\oplus\\
		&=\bigtimes_{i=1}^N \partial_i \bigl(G + F^\oplus\bigr)
		= \bigtimes_{i=1}^N \partial_i \dual.
\end{align*}
This concludes the proof.
\end{proof}

Inspired by the Sinkhorn iterations for $\mathrm{MOT}_\eps$ in \cite{MG2019} 
and $\mathrm{UOT}_\eps$ in \cite{SFVP19}, we propose Algorithm~\ref{def:sinkhorn} for solving $\mathrm{UMOT}_\eps$ in its dual form \eqref{eq:dual_problem_reg_mm_ub_OT}.
By Proposition~\ref{prop:sinkhorn_estimate} every fixed point of the sequence 
$(f^{(rN)})_{r \in \N}$ generated by Algorithm~\ref{def:sinkhorn}
is a solution of \eqref{eq:dual_problem_reg_mm_ub_OT}.
\begin{remark}
It holds $\varphi^*(0)=0$.
Hence, we can choose $f^{(0)}=0$ as an initialization with $\dual(f^{(0)}) > -\infty$.
\end{remark}

Next, we want to show that the sequence converges. Note that in \cite[Thm.~4.7]{MG2019} convergence of the (rescaled) Sinkhorn algorithm was 
shown by exploiting the property $\dual(f_1,\dotsc,f_N) = \dual(f_1 + \lambda_1,\dotsc,f_N + \lambda_N)$ 
for all $\lambda_1,\dotsc,\lambda_N \in \mathbb R$ with $\sum_{i=1}^N \lambda_i = 0$, which holds 
exclusively in the balanced case where $\varphi_i^*(q) = q$.
Hence, significant modifications  of the proof are necessary.
Albeit taking several ideas from \cite{SFVP19}, our approach differs as we cannot rely on the $1$-Lipschitz continuity of the $(c,\eps)$-transform, which only holds for $N=2$.
Instead, we exploit the compactness of the Sinkhorn operator as in \cite{MG2019}, for which we need to establish uniform boundedness of the iterates.
To this end, we need the following two lemmata.

\begin{lemma}\label{thm:D_eps_coercive}
Let $f^{(n)} \in L^{\infty, \times} (\X, \gamma)$, $n \in \N$,
satisfy $\Vert (f^{(n)})^\oplus \Vert_{\infty} \overset{n\to\infty}{\longrightarrow} \infty$ and have uniformly bounded oscillations (see Lemma~\ref{prop:properties_entropic-c-trafo}).
Then, it holds  $\dual(f^{(n)}) \rightarrow -\infty$.
\end{lemma}

\begin{proof}
    Since the entropy functions $\varphi_i$, $i=1,\ldots,N$, satisfy $\varphi_i(1) = 0$, we have $\varphi_i^*(x) \geq x$ for all $x \in \R$.
    Hence, we can estimate
    \begin{align*} - \sum_{i=1}^N \bigl \langle \mu_i, \varphi_i^*(-f^{(n)}_i) \bigr \rangle &\leq \langle \mu^\otimes, (f^{(n)})^\oplus \rangle .
    \end{align*}
    Since the $\mu_i$ are absolutely continuous with respect to $\gamma_i$ with density $\sigma_{\mu_i}$, we obtain
    \begin{align*}
    &\dual(f^{(n)}) \leq \biggl \langle \mu^\otimes, (f^{(n)})^\oplus\\ 
    &-\eps \exp\biggl(\frac{ (f^{(n)})^\oplus - c - \eps\ln\bigl(\prod_{i=1}^N \sigma_{\mu_i}\bigr)}{\eps}\biggr)\biggr\rangle.
    \end{align*}
    Clearly, $(f^{(n)})^\oplus$ has uniformly bounded oscillation.
    Hence, for $\Vert (f^{(n)})^\oplus \Vert_{\infty} \to \infty$ the integrand diverges to $-\infty$ on a set of positive measure, which yields the assertion.
\end{proof}

\begin{lemma}\label{lem:sej_lem5}
Let $\ln(\sigma_{\mu_i}) \in L^\infty(\X_i,\gamma_i)$, $i = 1,\dotsc,N$,
and $\dual(f^{(0)}) > -\infty$. For the Sinkhorn sequence $(f^{(n)})_{n\in \N} \subset L^{\infty,\times}(\X,\gamma)$
generated by Algorithm~\ref{def:sinkhorn}, there exists a constant $M>0$ and a sequence  \smash{$(\lambda^{(n)})_{n\in \N}$,  \smash{$\lambda^{(n)} \in \R^N$},}
with \smash{$\sum_{i=1}^N \lambda_i^{(n)} = 0$} such that
\[\bigl \Vert f_i^{(n)} + \lambda_i^{(n)} \bigr \Vert_{ \infty} < M.\]
\end{lemma}

\begin{proof}
For $i = 1,\dotsc,N-1$, $j = 0,\dotsc,N-1$ and $r \in \N$ set
\begin{align*}
&\lambda^{(rN + j)}_i \coloneqq \begin{cases} \aprox^\eps_{\varphi^*_i}(\lambda_{f^{(rN+i-1)},i}) & \!\!\!\!\!\text{if } i \leq j,\\  \aprox^\eps_{\varphi^*_i}(\lambda_{f^{((r-1)N+i-1)},i})& \!\!\!\!\!\text{if } i > j ,
\end{cases}\\
&\lambda_N^{(rN + j)} \coloneqq - \sum_{i=1}^{N-1} \lambda_i^{(rN + j)},
\end{align*}
where $\lambda_{f,i}$ is defined as in Lemma \ref{prop:properties_entropic-c-trafo}.
Since $\aprox$ is 1-Lipschitz, 
by Lemma~\ref{prop:properties_entropic-c-trafo}~i) 
and the definition of the iterates, for $i \leq j$ we obtain that
\begin{align}
&\bigl\Vert f^{(rN+j)}_i + \lambda_i^{(rN+j)} \bigr\Vert_{\infty}\\
={}&\bigl\Vert \bigl( f^{(rN+i-1)} \bigr)^{(c,\eps,\varphi,i)} +\lambda_i^{(rN+j)} \bigr\Vert_{ \infty} \\
={}&\bigl\Vert  -\aprox^\eps_{\varphi^*_i}\bigl(-(f^{(rN+i-1)} )^{(c,\eps,i)}\bigr) \\
&+\aprox^\eps_{\varphi^*_i}(\lambda_{f^{(rN+i-1)},i}) \bigr\Vert_{ \infty} \\
\leq{} &\bigl\Vert  ( f^{(rN+i-1)} )^{(c,\eps,i)} 
+ \lambda_{f^{(rN+i-1)},i} \bigr\Vert_{\infty} \hspace*{-1cm} \\
\leq {}&\Vert c \Vert_{\infty} + \sup_i \eps \Vert \ln (\sigma_{\mu_i}) \Vert_{ \infty}
\leq M_1
\end{align}
for some $M_1<\infty$.
In the case $N>i>j$, we have
$$f^{(rN+j)}_i=\bigl( f^{((r-1)N+i-1)} \bigr)^{(c,\eps,\varphi,i)},$$
such that this case works similarly.
Thus, it remains to estimate the last component.
By Proposition \ref{prop:sinkhorn_estimate} it holds for $n \in \N$ that \smash{$\dual(f^{(0)}) \leq \dual(f^{(n)})$}.
Due to Lemma~\ref{thm:D_eps_coercive}, this ensures the existence of $M_2 > 0$ such that \smash{$\Vert \bigoplus_{i=1}^N f^{(n)}_i \Vert_{\infty} \leq M_2$} for all $n \in \N$.
Thus,
\begin{align*}
    &{} \bigl \Vert f^{(rN+j)}_N + \lambda^{(rN+j)}_N \bigr \Vert_{\infty}\\
    \leq {}&  \Bigl \Vert \bigoplus_{i=1}^{N} f^{(rN+j)}_i \Bigr \Vert_{\infty}\\
    &+ \Bigl \Vert \lambda^{(rN+j)}_N - \bigoplus_{i=1}^{N-1} f^{(rN+j)}_i \Bigr \Vert_{\infty}\\
    \leq {}& M_2 + \sum_{i=1}^{N-1} \bigl \Vert f^{(rN+j)}_i + \lambda_i^{(rN+j)} \bigr \Vert_{\infty}\\
    \leq {}&M_2 + (N-1)M_1.
\end{align*}
For $M \coloneqq M_2 + (N-1)M_1$ the assertion follows.
\end{proof}

Now, we can prove convergence of the Sinkhorn iterates under mild additional assumptions on the entropy functions.
\begin{theorem}\label{thm:sink_converges}
Let
$\ln(\sigma_{\mu_i}) \in L^\infty(\X_i,\gamma_i)$, $i = 1,\dotsc,N$. 
Assume that $[0, \infty)\subset \dom \varphi_i$ for all $i=1, \dots, N-1$, and that $\dual(f^{(0)}) > -\infty$.
Then, the sequence $(f^{(n)})_{n \in \N}$ induced by Algorithm~\ref{def:sinkhorn} satisfies $\Vert (f^{(n)})^\oplus - \hat{f}^\oplus \Vert_p \xrightarrow[]{n \to \infty} 0$ for any optimal solution $\hat f \in L^{\infty,\times}(\X,\gamma)$ of \eqref{eq:dual_problem_reg_mm_ub_OT} for every $p \in [1,\infty)$.
Moreover, if $N-1$ of the $\varphi_i^*$ are strictly convex, then the optimal $\hat{f}$ is unique and $\Vert f^{(n)} - \hat{f} \Vert_p \to 0$ for every $p \in [1,\infty)$.
\end{theorem}
\begin{proof}
First, we show that the sequence $(f^{(n)})_{n \in \N}$ is uniformly bounded.
By Lemma~\ref{lem:sej_lem5}, 
there exists a sequence $(\lambda^{(n)})_{n\in \N} \subset \R^N$, with \smash{$\sum_{i=1}^N \lambda_i^{(n)} = 0$} and $M > 0$ such that
\begin{equation} \label{eq:help}
\bigl\Vert f_i^{(n)} + \lambda_i^{(n)} \bigr\Vert_{\infty} \leq M
\end{equation}
for all $i=1,\ldots,N$, $n\in \N$.
Define $g_i^{(n)} \coloneqq f_i^{(n)} + \lambda_i^{(n)}$.
To obtain uniform boundedness,
it suffices to show that
\smash{$\max_i |\lambda_i^{(n)}|$} is uniformly bounded in $n$.
We have for any $q_i \in \dom{\varphi_i^*}$ and $p_i \in \partial  \varphi_i^*(q_i)$ by the first order convexity condition and since elements of $\partial \varphi_i^*(q_i)$ are non-negative, that
\begin{align*}
    &\bigl \langle \mu_i, - \varphi_i^*(-f^{(n)}_i)\bigr \rangle\\
		= {}&\bigl \langle \mu_i, - \varphi_i^*(-g_i^{(n)} + \lambda_i^{(n)})\bigr \rangle\\ 
		\leq {}&\Bigl\langle \mu_i, -\varphi_i^*(q_i) + p_i \bigl(g_i^{(n)} - \lambda_i^{(n)} + q_i\bigr)\Bigr \rangle
		\\
    \leq {}&\Bigl \langle \mu_i, - \varphi_i^*(q_i) + p_i \bigl(\bigl\Vert g_i^{(n)}
		\bigr\Vert_{\infty} - \lambda_i^{(n)} + q_i\bigr) \Bigr \rangle.
\end{align*}
Consequently, we obtain by \eqref{eq:help} that
\begin{align*}
    \sum_{i=1}^N \bigl \langle \mu_i, - \varphi_i^*(-f_i^{(n)})\bigr \rangle
        \leq
        -\sum_{i=1}^N \mu_i (\X_i) p_i \, \lambda_i^{(n)}\\
   +\underbrace{\sum_{i=1}^N  \mu_i(\X_i) \bigl(p_i (M + q_i)-\varphi_i^*(q_i)\bigr)}_{=:K(p_i, q_i)}.
\end{align*}
Setting $m_i \coloneqq \mu_i (\X_i)$, $i=1,\ldots,N$, and using that the $\lambda^{(n)}_i$ sum up to zero, we conclude 
\begin{align} \label{now}
   &\sum_{i=1}^N \bigl \langle \mu_i, - \varphi_i^*(-f_i^{(n)})\bigr \rangle \nonumber \\
        \leq
       &\sum_{i=1}^{N-1} \bigl(m_N p_N- m_i p_i) \lambda^{(n)}_i + K(p_i, q_i).\quad
\end{align}
First, since $\dom \varphi_N^* =(-\infty, (\varphi_N)'_\infty)$ with $(\varphi_N)'_\infty \geq 0$, we can fix some $q_N\in \mathrm{int}(\dom \varphi_N^*)$ and some $p_N\in \partial \varphi_N^*(q_N)\neq \emptyset$.
Assume that there exists at least one $i \in \{1,\ldots,N\}$ 
such that $\smash{(\lambda_i^{(n)})_{n \in \N}}$ is unbounded.
Since the $\smash{\lambda_i^{(n)}}$ sum to zero, there then also exists at least one such $i \in \{1,\ldots,N-1\}$.
For each of these $i$, we can extract a subsequence $(n_{k})_k$ such that either \smash{$\lambda_i^{(n_k)} \to -\infty$} or \smash{$\lambda_i^{(n_k)} \to \infty$}.
In the first case, choose some $p_i > 0$ small enough such that $m_N p_N > m_i p_i$.
By assumption, $p_i\in \mathrm{int}(\dom \varphi_i)$, such that there exists $q_i\in \partial \varphi_i(p_i)$.
Since $\varphi\in \Gamma_0(\R)$, it follows that $p_i\in \partial \varphi_i^*(q_i)$.
Moverover, since
\[
\infty > p_i q_i = \varphi_i(p_i) + \varphi_i^*(q_i)
\]
and $p_i\in \dom \varphi_i$, we also have $q_i\in \dom(\varphi_i^*)$.
Similarly, if \smash{$\lambda_i^{(n_k)} \to \infty$}, choose some $p_i<\infty$ such that $m_N p_N < m_i p_i$ and $q_i\in \dom \varphi_i^*$ with $p_i\in \partial \varphi_i^*(q_i)$.
Now, we have by Proposition~\ref{prop:sinkhorn_estimate} that $\dual(f^{(0)}) \leq \dual(f^{(n_{k})})$.
Since $(f^{(n_k)})^\oplus = (g^{(n_k)})^\oplus$ and the \smash{$g_i^{(n_k)}$}, $i=1,\ldots,N$,
are uniformly bounded, the second summand in $\dual(f^{(n_k)})$ remains bounded as $k\rightarrow \infty$, while the first summand in $\dual(f^{(n_k)})$ goes to $-\infty$ by \eqref{now} with the above chosen $(q_i, p_i)$. 
This is a contradiction to $\dual(f^{(0)}) > -\infty$. 
Thus, there is $\tilde{M}>0$ such that for $i \in \{1,\dotsc,N\}$ and $n \in \N$ it holds
\begin{align*}
\bigl\Vert f_i^{(n)}\bigr\Vert_{\infty}
\leq {}&\bigr \Vert f_i^{(n)} + \lambda_i^{(n)}\bigl \Vert_{\infty}
+ \bigl \vert \lambda_i^{(n)} \bigr \vert\\
\leq {}&M + \tilde{M}.
\end{align*}
Hence, $(f^{(n)})_n$ is a uniformly bounded sequence.
By Lemma~\ref{prop:properties_entropic-c-trafo} ii), we know that the operator
$S^{(c,\eps,\varphi,j)} \colon L^{\infty,\times}(\X,\gamma)\to L^p(\X_j,\gamma_j)$ 
is compact for every $j$ and $p \in [1,\infty)$.
Consequently, we get existence of a converging subsequence $(f^{(n_k)})_{k \in \N}$ in $L^{p,\times}(\X,\gamma)$.
As $f^{(n_k)}$ is uniformly bounded in $L^{\infty,\times}(\X,\gamma)$ and since $\infty$-norm balls are closed under $L^p$ convergence, we get that the limit $\hat f$ additionally satisfies $\hat f \in L^{\infty,\times}(\X,\gamma)$.

Now, we prove optimality of $\hat{f}$.
Note that there is $j \in \{0,\dotsc,N-1\}$ so that $n_k \equiv j \mod N$ for infinitely many $k \in \N$.
Without loss of generality assume $j = 0$.
Then, we restrict $(n_k)_k$ to $n_k \equiv 0 \mod N$ for all $k \in \N$.
Using the Lipschitz continuity of $\ln$ and $\exp$ on compact sets and the 1-Lipschitz continuity of $\aprox$ together with the uniform boundedness of the sequence $(f^{(n)})_{n \in \N}$, we obtain that
{\allowdisplaybreaks
\begin{align*}
&\bigl\Vert \bigl(f^{(n_k)}\bigr)^{(c,\eps,\varphi,j)} - \hat{f}^{(c,\eps,\varphi,j)} \bigr\Vert^p_{p} \\
\leq {}
&\bigl\Vert \bigl(f^{(n_k)}\bigr)^{(c,\eps,j)} - \hat{f}^{(c,\eps,j)} \bigr\Vert^p_{p}
\\
\leq {}&C \,  \biggl\Vert 
\int_{\X_{\mj}} \exp\biggl(\frac{(f^{(n_k)})_{\mj}^\oplus - c}{\eps}\biggr) \\
&\quad\,\,\,- \exp\biggl(\frac{{\hat f}_{\mj}^\oplus - c}{\eps}\biggr) \dx \gamma_{\mj}^\otimes \biggr\Vert^p_{p} \\
\leq{}& C \int_{\X} \bigl\vert (f^{(n_k)})_{\mj}^\oplus - {\hat f}_{\mj}^\oplus \bigr \vert^p \dx \gamma^\otimes \\
\leq{}&C \bigl\Vert f^{(n_k)} - \hat f \bigr\Vert^p_{p}\to 0
\end{align*}}
for every $j=1,\ldots,N$, where $C$ stands for some unspecified, possibly changing constant.
In particular, it holds
\begin{align*}
f_1^{(n_k+1)} \to &\hat{f}^{(c,\eps,\varphi,1)} \in L^p(\X_1, \gamma_1).
\end{align*}

As all $\varphi_i^*$ are lsc, the dominated convergence theorem implies that $\dual(f^{(n_j)}) \to \dual(\tilde f)$ for any a.e.\ convergent subsequence $(f^{(n_j)})_{j \in \N}$ of $(f^{(n)})_{n \in \N}$  with limit $\tilde f$.
Due to this continuity property and since $(\dual(f^{(n)}))_n$ is a convergent sequence, we get
\begin{align*}
    &\dual\bigl(\hat{f}^{(c,\eps,\varphi,1)},\hat{f}_2,\dotsc,\hat{f}_N\bigr)\\
    ={}& \lim_{k \to \infty} \dual\bigl(f^{(n_k+1)}\bigr)
    = \lim_{k \to \infty} \dual\bigl(f^{(n_k)}\bigr)\\
    ={} &\dual(\hat{f}).
\end{align*}
Consequently, Proposition~\ref{prop:sinkhorn_estimate} implies $\hat{f}_1 = \hat{f}^{(c,\eps,\varphi,1)}$.
In the same way, we obtain
\begin{align*}
    \dual\bigl(\hat{f}^{(c,\eps,\varphi,1)},\hat{f}^{(c,\eps,\varphi,2)},\hat{f}_3\dotsc,\hat{f}_N\bigr)
    = \dual(\hat{f}).
\end{align*}
Due to $\hat{f}_1 = \hat{f}^{(c,\eps,\varphi,1)}$, this gives $\hat{f}_2 = \hat{f}^{(c,\eps,\varphi,2)}$. Proceeding iteratively this way, we obtain that
$\hat{f}_i = \hat{f}^{(c,\eps,\varphi,i)}$ for all $i=1, \dots, N$.
Hence, Proposition~\ref{prop:sinkhorn_estimate} implies that $\hat{f}$ is an optimal dual vector.

If $N-1$ of the $\varphi^*$ are strictly convex, then the maximizer in \eqref{eq:dual_problem_reg_mm_ub_OT} is unique and $(f^{(n)})_{n \in \N}$ converges.
Otherwise, we obtain convergence of $((f^{(n_j)})^\oplus)_{j \in \N}$ to $\hat{f}^\oplus$ in $L^p(\X,\gamma^\otimes)$.
Since $\hat{f}^\oplus$ is the same for all possible limit points $\hat{f}$, convergence of $((f^{(n)})^\oplus)_{n \in \N}$ follows.
\end{proof}

\begin{remark} [Relation to UOT in \cite{SFVP19}] 
For entropy functions satisfying the assumptions of Theorem \ref{thm:sink_converges}, our result can be seen as a generalization of the UOT result in Séjourné et al.\ \cite{SFVP19}
to the multi-marginal case and non-Lipschitz costs.
Notably, the result by Séjourné et al.\ covers some additional entropy functions.
\end{remark}
Finally, we want to remark that
all results of this section also hold true if we do not assume that the spaces $\X_i$, $i=1,\ldots,N$, are compact as long as the cost function $c$ remains bounded.

\section{Barycenters and Tree-Structured Costs}\label{sec:BaryInt}
In this section, we are interested in the computation of $\mathrm{UOT}_\eps$ barycenters 
and their relation to $\mathrm{UMOT}_\eps$
with tree-structured costs.
An undirected graph $\G = (\V,\E)$ with $N$ nodes 
$\V= \{1,\ldots,N\}$ and edges $\E$ is a \emph{tree} if it is acyclic and connected. We write $e = (j,k)$ if $e$ joins the nodes $j$ and $k$, 
where we agree that $j < k$ in order to count edges only once.
Let $\deg(i)$ denote the number of edges in node $i \in \mathcal V$.
A node $i \in \mathcal V$ is called a \emph{leaf} if $\deg(i) = 1$. 
By $\mathcal N_j$ we denote the set of neighbors of node $j$.
For a given tree, let
$t = (t_e)_{e \in \E}$ with $t_{e} \in [0,1]$ and $\sum_{e \in \E} t_e = 1$. 
A cost function $c_t$ is said to be \emph{tree-structured}, if it is of the form
\begin{equation}\label{eq:tree-structured-costs}
    c_t = \sum_{e \in \E} t_e c_e  = \sum_{(j,k) \in \E} t_{(j,k)} c_{(j,k)}.
\end{equation}
In Section~\ref{sec:barycenters}, we consider the case where
the tree is star-shaped, i.e., $\E=\{ (1, N), \dots, (N-1, N) \}$, see Fig.~\ref{fig:trees} left.
General tree-structured costs as, e.g., those in Fig.~\ref{fig:trees} middle and right 
are addressed in Section~\ref{subsec:tree-costs}.
We restrict our attention to $\X_i$, $i=1,\ldots,N$, which are either finite or compact subsets of $\R^{d_i}$.
Moreover, all references measures $\gamma_i$, $i=1,\ldots,N$, are counting measures, respectively Lebesgue measures, so that we regularize exclusively with the entropy from Rem.~\ref{rem:kl} ii) or iii), respectively.

\begin{figure*}[t!]
	\centering
	\includegraphics[width=0.2\textwidth]{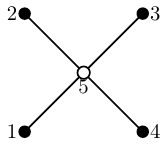} \hspace{1cm}
	\includegraphics[width=0.2\textwidth]{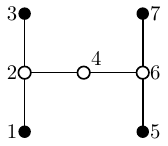} \hspace{1cm}
	\includegraphics[width=0.037\textwidth]{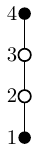}
	
	\caption{Examples of trees: star-shaped (left), H-shaped (middle), line-shaped (right).
	}
	\label{fig:trees}
\end{figure*}	

\subsection{Barycenters}\label{sec:barycenters}
For the barycenter problem with respect to $\mathrm{UOT}_\eps$, we introduce an additional finite set $\Y$ (or compact set $\Y \subset \R^{d_{N+1}}$), and use for $\xi \in \mathcal M^+(\Y)$ the entropy function
\begin{equation} \label{eq:psi}
\psi \coloneqq \iota_{\{1\}}, \quad D_\psi(\cdot, \xi) = \iota_{\{\xi\}},
\end{equation}from Example~\ref{ex:1} i).
Let $c_i \in  C(\X_i \times \Y)$, $i=1,\ldots,N$, be non-negative cost functions.
The corresponding tree is star-shaped, i.e., given by $\V=\{ 1, \dots, N+1\}$ and $\E=\{ (1, N+1), \dots, (N,N+1) \}.$
To emphasize the dependence of $\UOT_\eps$ on these functions, we write
$\mathrm{UOT}_\eps^{(c_i,\varphi_i,\psi)}$.
We use an analogous notation for $\UMOT_\eps$.
Let 
\[
\Delta_N \coloneqq \Bigl\{t = (t_i)_{i=1}^N: t_i \ge 0, \, \sum_{i=1}^N t_i = 1\Bigr\}
\]
be the $(N-1)$-dimensional probability simplex.
For given barycentric coordinates $t \in \Delta_N$, the \emph{barycenter} $\hat{\xi} \in \M^+(\Y)$ 
of $\mu$ with respect to $\mathrm{UOT}_{\eps}$ 
is given by
    \begin{align}\label{eq:ub_bary}
     \hat \xi \coloneqq &\argmin_{\xi \in \M^+(\Y)} 
		\sum_{i=1}^N t_i \UOT_{\eps/t_i}^{(c_i,\varphi_i,\psi)}(\mu_i,\xi),\\
		=   &\argmin_{\xi \in \M^+(\Y)}
		\sum_{i = 1}^N  \min_{\pi^{(i)} \in \mathcal M^+(\X_i \times \Y)} 
		\int_{\X_i \times \Y}\!\! t_i c_i \dx \pi^{(i)}\notag\\
		&+ \eps E \bigl(\pi^{(i)}\bigr) + t_i D_{\varphi_i}\bigl(P_{\X_i} \strut{}_\#\pi^{(i)} , \mu_i\bigr), \\
		& \; \mathrm{subject \; to} \quad P_{\Y}\strut{}_\#\pi^{(i)} = \xi, \; i = 1,\ldots,N. \nonumber
\label{eq:ub_bary_1}		
    \end{align}
Note that by the choice of $\psi$  the barycentric marginal $\xi$ is exactly matched.
By Proposition \ref{prop:ub_existence}, the involved $\UOT$ problems have unique solutions.
Moreover, it was shown in \cite[Sec.\ 5.2]{CPSV17} that a unique barycenter exists due to the regularization.
However, these barycenters do not correspond to a shortest path since $\UOT_\eps$ is not a metric.

To establish a relation with the multi-marginal setting,
 we exploit that the optimal plan of the multi-marginal problem with cost function
\begin{equation}\label{eq:cost}
c_t \coloneqq \sum_{i=1}^N t_i c_i \in C(\X \times \Y)
\end{equation}
is readily determined by its marginals.
We need the following auxiliary lemma.

\begin{lemma}\label{lem:bary_lemma}
Let $\pi \in \M^+(\X \times \Y)$ be absolutely continuous with respect to Lebesgue measure, respectively counting measure, on $\X \times \Y$ with density $\sigma_\pi$.
If there exists $a_i \in \C(\X_i \times \Y)$, $i = 1,\ldots,N$, such that 
$\sigma_{\pi} = \prod_{i=1}^N a_i$, 
then $\pi$ is related to its marginals $\pi_i \coloneqq {P_{\X_i \times \Y}}_\# \pi$, $i = 1,\ldots,N$, and $\pi_{N+1} \coloneqq {P_{\Y}}_\# \pi$ via
\[\sigma_{\pi} = \frac{\sigma_{\pi_{1}} \cdots \sigma_{\pi_{ N}}}{{\sigma^{N-1}_{\pi_{N+1}}}},\]
where $\sigma_{\pi_i}$ denotes the density of $\pi_i$ with respect to the Lebesgue, respectively counting measure, on $\X_i \times \Y$.
\end{lemma}

\begin{proof}
By abuse of notation we denote the Lebesgue measure, respectively counting measure, by $\lambda$. The underlying space becomes clear from the context.
By assumption, the marginal densities of $\pi$ read
\[
\sigma_{\pi_j}(x_j,y) = a_j(x_j,y) \prod_{i \neq j} \int_{\X_i} a_i(x_i, y) \dx \lambda(x_i),
\]
$j = 1,\dotsc,N$.
Since $P_{\Y}\strut{}_\#  \pi = \pi_{N+1}$, 
we have 
\smash{$\sigma_{\pi_{N+1}} = \prod_{i = 1}^N \int_{\X_i} a_i (x_i, \cdot) \dx \lambda(x_i)$}, 
which finally yields
\begin{align*}
    \sigma_{\pi} &= \prod_{i=1}^N a_i\\
		&= 
		\frac{ \sigma_{\pi_{1} } \ldots \sigma_{\pi_{N}} }
		{
		\prod_{j=1}^N
		\Bigl(\prod_{i \neq j} \int_{\X_i} a_i (x_i, \cdot)\dx \lambda_i(x_i)\Bigr)
		}\\ 
		&= 
		\frac{ \sigma_{\pi_{1}} \ldots \sigma_{\pi_{N}} }
		{\sigma^{N-1}_{\pi_{N+1}}}.
\end{align*}
\end{proof}
Now, we can draw the aforementioned connection between the barycenter problem \eqref{eq:ub_bary} and $\UMOT_\eps^{(c_t,t\varphi,\varphi_{N+1})}$, where
\begin{align} \label{eq:tildepsi}
\varphi_{N+1} &\coloneqq \iota_{[0,\infty)},\\
D_{\varphi_{N+1}}(\cdot, \xi) &= 0, \quad \xi \in \M^+(\Y),
\end{align}
see Example \ref{ex:1}ii).
Due to the special form of this entropy, the $(N+1)$-th input of $\UMOT_\eps^{(c_t,t\varphi,\varphi_{N+1})}$ has no effect on the functional.
To emphasize this, we use the notation
\begin{align}\label{eq:xxl}
&\UMOT_\eps^{(c_t,t\varphi,\varphi_{N+1})}
(\mu,\ast)\\
\coloneqq &\min_{ \pi \in \mathcal M^+(\X \times \Y) }
\int_{\X \times \Y} c_t \dx \pi + \eps E(\pi)
		\\ &+ 
		\sum_{i=1}^N D_{t_i \varphi_i}(P_{\X_i} \strut{}_\#\pi, \mu_i).
\end{align}
The next theorem establishes the relation between the barycenter problem and this multi-marginal optimal transport problem.
\begin{theorem}\label{thm:bary_mot_relation}
For $i = 1,\ldots,N$,
let $\mu_i \in \M^+(\X_i)$ 
and cost functions $c_i\in C(\X_i \times \Y)$ be given.
Define $c_t \in C(\X \times \Y)$ by \eqref{eq:cost}, $\psi$ by \eqref{eq:psi}  and $\varphi_{N+1}$  by 
\eqref{eq:tildepsi}. 
If $\hat{\pi}$ is optimal for \eqref{eq:xxl}
then $P_{\Y} \strut{}_\# \hat{\pi}$ minimizes the functional $F(\xi)$ given by
\begin{align}\label{eq:bary_mot_relation}
   \sum_{i=1}^N t_i \UOT_{\eps/t_i}^{(c_i,\varphi_i,\psi)}(\mu_i,\xi) - \eps(N-1) E(\xi).\quad
\end{align}
\end{theorem}

\begin{proof}
Again, let $\lambda$ denote the the Lebesgue or counting measure, respectively.
Set
\[
C = \eps \lambda(\Y) \biggl(\sum_{i=1}^N \lambda(\X_i) - (N-1) - \prod_{i=1}^N \lambda(\X_i)\biggr).
\]
Estimating in both directions, we show that
\begin{align*}
&\inf_{\xi \in \M^+(\Y)} F(\xi)
= \UMOT_\eps^{(c_t,t\varphi,\varphi_{N+1})}(\mu,\ast) + C. 
\end{align*}

1. Fix $\xi \in \M^+(\Y)$ such that optimal plans $\hat \pi^{(i)} \in \M^+(\X_i \times \Y)$ 
for $\UOT_\eps^{(c_i,\varphi_i,\psi)}(\mu_i,\xi)$, $i = 1,\dotsc,N$, exist.
Then, we define
\[
\pi \coloneqq  \frac{\sigma_{\hat \pi^{(1)}} \cdots \sigma_{\hat \pi^{(N)}}}{{\sigma_\xi}^{N-1}} 
\, \lambda^\otimes \in \M^+(\X \times \Y),
\]
which yields $P_{\X_i \times \Y} \strut{}_\# \pi = \hat \pi^{(i)}$ and $P_{\Y} \strut{}_\# \pi = \xi$.
Consequently, we get
\begin{align*}
&F(\xi)
=\sum_{i = 1}^N \biggl(\int_{\X_i \times \Y} t_i c_i \dx \hat \pi^{(i)} + \eps E\bigl(\hat \pi^{(i)}\bigr)\biggr) \\
		&+ \sum_{i = 1}^N t_i D_{\varphi_i}\bigl(P_{\X_i}\strut{}_\# \hat \pi^{(i)} , \mu_i\bigr) - \eps (N-1) E(\xi)\\
		=&
			\int_{\X \times \Y} \!\!c_t \dx \pi + 
			\eps \Bigl( \sum_{i = 1}^N E\bigl(\hat \pi^{(i)}\bigr) \!-\! (N-1) E(\xi) \Bigr)
		\\
		&+ \sum_{i = 1}^N t_i D_{\varphi_i}\bigl(P_{\X_i} \strut{}_\# \hat \pi^{(i)} , \mu_i\bigr).
\end{align*}
Using the definition of $\pi$, we obtain
{\allowdisplaybreaks
\begin{align*}
&\sum_{i = 1}^N E\bigl(\hat \pi^{(i)}\bigr) - (N-1) E(\xi)\\
= {}&\sum_{i = 1}^N \biggl(\int_{\X_i \times \Y} \ln\bigl(\sigma_{\hat \pi^{(i)}}\bigr) \dx \hat \pi^{(i)} + \lambda(\X_i \times \Y)\biggr)\\
&- (N-1) \biggl(\int_{\Y} \ln(\sigma_\xi) \dx \xi + \lambda(\Y)\biggr)\\
&\underbrace{-\sum_{i=1}^N \pi^{(i)}(\X_i \times \Y) + (N-1)\xi(\Y)}_{= -\pi(\X \times \Y)}\\
= &\int_{\X \times \Y} \ln(\sigma_\pi) \dx \pi - \pi(\X \times \Y)\\
&+ \sum_{i=1}^N \lambda(\X_i \times \Y) - (N-1)\lambda(\Y)\\
= {}&E(\pi) + C/\eps.
\end{align*}
}
Incorporating $t_i D_{\varphi_i} = D_{t_i \varphi_i}$, we obtain
\begin{align*}
 F(\xi)  
		= &\int_{\X \times \Y} c_t \dx \pi + \eps E(\pi)+C \\
		&+ \sum_{i=1}^N D_{t_i \varphi_i}(P_{\X_i} \strut{}_\# \pi, \mu_i) .
\end{align*}		
Thus, minimizing the right hand side over all $\pi \in \M^+(\X \times \Y)$,
we get
\[
F(\xi) \geq \UMOT_\eps^{(c_t,t\varphi,\varphi_{N+1})}(\mu,\ast) + C,
\]
such that minimizing
the left hand side over all $\xi \in \M^+(\Y)$ yields the desired estimate.

2. Next, we show the converse estimate. 
Let $\hat \pi \in \M^+(\X \times \Y)$ be the optimal plan and $(\hat f, \hat g) \in L^{\infty,\times}(\X) \times \C(\Y)$ be optimal dual potentials for $\UMOT_\eps^{ (c_t,t\varphi,\varphi_{N+1})}\left(\mu,\ast\right)$. 
For $i = 1,\dotsc,N$, define $a_i \in \C(\X_i \times \Y)$ by
\begin{align*}
a_1 &\coloneqq \exp\Bigl(\frac{\hat f_1 + \hat g - t_1 c_1}{\eps}\Bigr) \quad \text{ and }\\
a_i &\coloneqq
\exp \Bigl( \frac{\hat f_i - t_i c_i}{\eps}\Bigr), \quad i = 2,\dotsc,N.
\end{align*}
The definition of $c_t$ together with \eqref{eq:pd} yields
\[
\sigma_{\hat \pi} = \exp \Bigl( \frac{\hat f^\oplus + \hat g - c_t}{\eps} \Bigr) = \prod_{i=1}^N a_i.
\]
Then Lemma \ref{lem:bary_lemma} implies
\[
\hat \pi = 
(\sigma_{\hat \pi_{1}} \ldots \sigma_{\hat \pi_{N}}/{\sigma_\xi}^{N-1})\lambda^\otimes,
\]
where $\hat{\pi}_i \coloneqq {P_{\X_i \times \Y}}_\# \hat{\pi}$, $i = 1,\ldots,N$, and $\xi \coloneqq P_{\Y}\strut{}_\#\hat \pi$.
Similarly as for the previous considerations, this results in
\begin{align*}
    &\UMOT_\eps^{(c_t,t\varphi,\varphi_{N+1})}(\mu,\ast) + C\\
    = 
		{}&\sum_{i = 1}^N \biggl(
		\int_{\X_i \times \Y} t_i c_i \dx \hat \pi_{i} 
		+ 
		\eps E(\hat \pi_{i})\biggr)\\
		{}& + \sum_{i = 1}^N t_i D_{\varphi_i}\bigl(P_{\X_i}\strut{}_\#\hat \pi_{i} , \mu_i\bigr)
		- \eps (N-1) E(\xi)
    \\
    \geq 
		{}&
		F(\xi).
\end{align*}
Minimizing the right hand side over all $\xi \in \M^+(\Y)$ yields 
\begin{align}
    \! \UMOT_\eps^{(c_t,t\varphi,\varphi_{N+1})}(\mu,\ast) + C
		\geq
        \!\!\inf_{\xi \in \M^+(\Y)} \! F(\xi)
		\end{align}
and thus the desired equality.
As a direct consequence we get $P_{\Y}\strut{}_\# \hat \pi$ minimizes \eqref{eq:bary_mot_relation}. This concludes the proof.
\end{proof}

\begin{remark}\label{rem:equivalence_generalization}
The previous result directly generalizes to $\UMOT$ problems that also regularize the marginal on $\Y$.
More precisely, let $\mu_{N+1} \in \M^+(\Y)$ and $\varphi_{N+1} \in \Gamma_0(\R)$ be an arbitrary entropy function.
If $\hat{\pi}$ is the optimal plan for 
$\mathrm{UMOT}_\eps^{(c_t,t\varphi,\varphi_{N+1})}\left( \mu,\mu_{N+1}\right)$,
then $P_{\Y} \strut{}_\# \hat{\pi}$ solves
\begin{align}\label{eq:bary_mot_relation_2}
    &\min_{\xi \in \M^+(\Y)} \sum_{i=1}^N t_i \UOT_{\eps/t_i}^{(c_i,\varphi_i,\psi)}(\mu_i,\xi) \nonumber\\
    &- \eps(N-1) E(\xi) + D_{\varphi_{N+1}}(\xi,\mu_{N+1}).
\end{align}
\end{remark}

\begin{remark}[Comparison of formulations \eqref{eq:ub_bary} and \eqref{eq:xxl}]\label{rem:smear}
The proof of Theorem~\ref{thm:bary_mot_relation} reveals that the barycenter $\hat\xi$ in \eqref{eq:ub_bary} is ``over-regularized'', since it appears as the marginal measure of $\pi^{(i)}$ in each of the $N$ entropy terms $E(\pi^{(i)})$.
On the other hand, the proposed multi-marginal approach does not involve these $N-1$ superfluous regularizers.
This ensures that the minimizer $\hat \xi$ is less ``blurred'' compared to the original $\UOT_\eps$ barycenter, which is favorable for most applications.
A numerical illustration of this behavior is given in Section~\ref{sec:numerical_examples}.
We will see in Subsection~\ref{sec:tree_costs} that for tree-structured costs the computation of optimal transport plans for the multi-marginal case has the same complexity per iteration as for the ``barycentric'' problems.

Furthermore, the computation of barycenters with an additional penalty term as outlined in Rem.~\ref{rem:equivalence_generalization} is possible with the Sinkhorn-type algorithm detailed in Sec.~\ref{sec:tree_costs}.
In contrast, we are unaware of an efficient algorithm to solve the corresponding pairwise coupled formulation.

On the other hand, we only obtain the equivalence of the ``pairwise coupled'' formulation \eqref{eq:ub_bary} and the multi-marginal approach for the choices made in \eqref{eq:psi} and \eqref{eq:tildepsi}.
These enforce that all marginals of the plans in \eqref{eq:ub_bary} coincide with the barycenter, which is not necessary in general.
Although this generalization comes at the cost of a nested optimization problem, the pairwise coupled formulation is thus more flexible than $\UMOT$.



\end{remark}

\begin{remark}[Barycenters and MOT]\label{rem:bary-mot}
By Theorem \ref{thm:bary_mot_relation}, formula \eqref{eq:bary_mot_relation} with $\varepsilon = 0$ and $\varphi_i = \iota_{\{1\}}$, $i=1,\ldots,N$,
	our MOT formulation with the cost 
	$$c_t(x_1,\ldots,x_{N+1}) = \sum_i t_i c_i(x_i,x_{N+1})$$
	is equivalent to the OT barycenter problem.
	There is another reformulation of the OT barycenter problem
	via the McCann interpolation using the cost function
	$$
	c(x_1,\ldots,x_N) \coloneqq \min_y \sum_{i=1}^N t_i c_i(x_i,y)
	$$
	if a unique minimizer exists, see \cite{CE10}.
	To the best of our knowledge, there is no similar reformulation for our setting.
\end{remark}

\subsection{General Tree-Structured Costs}\label{subsec:tree-costs}
In the above barycenter problem, we have considered $\UMOT_\eps$ with a tree-structured cost function,
where the tree was just star-shaped.
In the rest of this section, we briefly discuss an extension of the
$\UMOT_\eps$ problem to costs of the form \eqref{eq:tree-structured-costs}, where $\G=(\V, \E)$ is a a general tree graph with cost functions $c_{(j,k)} \in C(\X_j\times \X_k)$ for all $(j, k)\in \E$.
For the balanced case, this topic was addressed in \cite[Prop.~4.2]{HRCK20}.


For a disjoint decomposition
$$
\mathcal V = V \cup U, \quad V \cap U = \emptyset,
$$ 
where $V$ contains only leaves,
and measures $\mu_v \in \M^+(\X_v)$, $v \in V$,
we want to find measures $\mu_u \in \M^+(\X_u)$, $u \in U$, that solve the problem
\begin{align}\label{ex:uot_graph}
    &\inf_{(\mu_u)_{u \in U}} 
		\sum_{(j,k) \in \E} t_{(j,k)} \UOT_{\eps/t_{(j,k)}}^{(c_{(j,k)}, \varphi_{j},\varphi_{k})}(\mu_{j},\mu_{k}).
\end{align}
Again, we assume that the unknown marginals $\mu_u$ 
are exactly matched, i.e., $\varphi_u = \iota_{\{1\}}$, $u \in U$.

\begin{example}   
For the barycenter problem~\eqref{eq:ub_bary}, we have $V = \{1,\ldots,N\}$ and the tree is star-shaped, meaning that \[
\mathcal E = \{(j,N+1): j=1,\ldots,N\},
\]
see Fig.~\ref{fig:trees} left.
In Fig.~\ref{fig:trees} middle, we have an H-shaped tree with $N=7$, edge set
$\mathcal E = \{(1,2),(2,3),(2,4),(4,6), (5,6),(6,7)\}$, and we consider problem \eqref{ex:uot_graph}
with $V = \{1,3,5,7\}$.
Finally, Fig.~\ref{fig:trees} right shows a line-shaped tree
with $N=4$, edge set $\mathcal E = \{(1,2),(2,3),(3,4)\}$ and $V = \{1,4\}$.
This graph is related to a so-called multiple barycenter problem,
and its solution was discussed for the balanced case, e.g., in \cite{Caillaud2020}.
\end{example}

In general, it is unclear how to solve problem \eqref{ex:uot_graph} using Sinkhorn iterations.
Therefore, we propose to solve a related multi-marginal problem
\begin{equation}\label{eq:mult_bary}
\UMOT_\eps^{\left(c_t, (t_v \varphi_v)_{v \in V}, (\varphi_u)_{u \in U} \!\right)} ((\mu_v)_{v \in V}, \!(\mu_u)_{u \in U})
\end{equation}
where again $\varphi_u \coloneqq \iota_{[0,\infty)}$ for all $u \in U$ 
and $t_v = t_{e}$ if $e$ in $\mathcal E$ joins $v$ with some other node (indeed well-defined for leaves).
Then, we can prove in analogy to Lemma~\ref{lem:bary_lemma} that the optimal plan $\hat \pi$ is related to its marginals $\hat \pi_e$  and $\hat \mu_u \coloneqq P_{\X_u}\strut{}_\# \hat \pi$ by
\[
\sigma_{\hat \pi} = 
\frac{\prod_{e \in \E} \sigma_{\hat \pi_e}}
{\prod_{u \in U} \sigma_{\hat \mu_u}^{\deg(u) - 1}}.
\]
Furthermore, we can show similarly as in the proof of Theorem \ref{thm:bary_mot_relation} the following corollary.
\begin{corollary}
Under the above assumptions, if $\hat\pi$ is the optimal plan in \eqref{eq:mult_bary}, then
the $\hat \mu_u= P_{\X_u}\strut{}_\# \hat \pi$, $u\in U$, solve
\begin{align*}
    \!\!\!\inf_{\mu_u \in \M^+(\X_u)} \!
		\sum_{(j,k) \in \E} \!\!\!t_{(j,k)}\! \UOT_{\eps/t_{(j,k)}}^{(c_{(j,k)}, \varphi_{j},\varphi_{k})}\!(\mu_{j},\mu_{k}) \\
		- \eps \sum_{u \in U} (\deg(u) - 1) E(\mu_u).
\end{align*}
\end{corollary}

\subsection{Efficient Sinkhorn Iterations for Tree-Structured Costs}\label{sec:tree_costs}
%
Throughout this section, let $\X_1,\dotsc,\X_N$ be finite subsets 
of $\R^d$ of size $M_i \coloneqq \lvert \X_i \rvert > 0$, $i = 1,\ldots,N$.
Furthermore, let $\mu_i \in \M^+(\X_i)$ be positive measures that are identified with vectors in $\R_+^{M_i}$. 
Hence, the reference measures $\gamma_i$ are chosen as the counting measure.
Recall that the Sinkhorn mapping for a cost function 
\smash{$c \in \R_{\geq 0}^{M_1 \times \dotsc \times M_N}$} is the concatenation of the $\aprox$ operator and the multi-marginal $(c,\eps)$-transform. 
As the former is applied pointwise, its computational cost is negligible.
Hence, it suffices to discuss the efficient implementation of the multi-marginal $(c,\eps)$-transform. 
For vectors and matrices, we denote pointwise multiplication by $\odot$ and pointwise division by $\oslash$.
Set 
\[
K \coloneqq  \exp(-c / \eps) \in \R_+^{M_1 \times \dotsc \times M_N}.
\]
For efficiency reasons, we perform computations in the $\exp$-domain, i.e., instead of the Sinkhorn iterates $(f^{(n)})_{n \in \N}$
in Algorithm \ref{def:sinkhorn} we consider
\begin{align*}u^{(n)} &\coloneqq u_1^{(n)} \otimes \dotsc \otimes u_N^{(n)}\\
u_i^{(n)} &\coloneqq \exp\Bigl(\frac{f^{(n)}_i}{\eps}\Bigr) , \quad i = 1,\dotsc,N.
\end{align*}
Convergence of Algorithm~\ref{def:sinkhorn} implies convergence of $u^{(n)}$ to some $\hat u \in \R_+^{M_1 \times \dotsc \times M_N}$.
By Proposition \ref{prop:duality_reg_mm_ub_OT}, the optimal plan $\hat{\pi}$ is given by $\hat{\pi} = K \odot \hat u$.
For $n \equiv j \mod N$ the Sinkhorn updates in the $\exp$-domain can be written as
\begin{align}\label{eq:sink_exp_domain}
u_j^{(n)} = &\exp \biggl( -\frac{1}{\eps} \aprox_{\varphi_i^*}^\eps\biggl(-\eps \ln \biggl(\\
&\frac{u_j^{(n-1)} \odot \mu_j}{P_{\X_j} \strut{}_\#(K \odot u^{(n-1)})}\biggr)\biggr) \biggr),
\end{align}
where division has to be understood componentwise.
Note that in this context $P_{\X_j}\strut{}_\#(\cdot)$ corresponds to summing over all but the $j$-th dimension.
Although the involved expression $\exp( -\frac{1}{\eps} \aprox_{\varphi_i^*}^\eps(-\eps \ln(\cdot)))$ appears to be complicated, it simplifies for all the entropies from Example~\ref{ex:2}, see also \cite{CPSV17}.

As recently discussed for the balanced case in \cite{HRCK20},  
multi-marginal Sinkhorn iterations can be computed efficiently 
if the cost function decouples according to a tree, see also  \cite{AB20} for a wider class of cost functions.
In this section, we generalize the approach  for tree-structured costs to the unbalanced setting. 
As in the balanced case, computing the projections $P_{\X_j}\strut{}_\#(K \odot u^{(n)})$, $j = 1,\dotsc,N$, is the computational bottleneck of the Sinkhorn algorithm.
Fortunately, the Sinkhorn iterations reduce to standard matrix-vector multiplications in our particular setting.

\begin{algorithm*}[t]
	\begin{algorithmic}
	    \State \textbf{Input:} Tree $\G = (\V,\E)$ with $N$ nodes, discrete measures $\mu_i$, $i=1,\dotsc,N$, and cost function $c$ decoupling according to $\G$
	    \State Choose a root $v_0$ and initialize $u_{j} = 1$ for all $j \in \mathcal V$
		\State $\V_{\rightarrow} \gets \text{PreOrderDepthFirstSearch}(\E,\V,v_0)$
		\For{$j\in \mathcal V_{\leftarrow} \setminus \{v_0\}$}
			    \State Initialize $\alpha_{(p(j), j)}$ according to \eqref{eq:alpha_updates}
			\EndFor
		\While{Sinkhorn not converged}
			\For{$j \in \V_{\rightarrow}$}
			    \If {$j \neq v_0$}
			        \State Update $\alpha_{(j, p(j))}$ according to \eqref{eq:alpha_updates}
			    \EndIf
			    \State Set $u_j \gets \exp \bigl(-\frac{1}{\eps}\aprox_{\varphi_i^*}^\eps \bigl( -\eps\ln\bigl(\mu_j \oslash \bigodot_{l \in \mathcal{N}_j} 
					\alpha_{(j,l)} \bigr)\bigr) \bigr)$
			\EndFor
			\For{$j\in \V_{\leftarrow} \setminus \{v_0\}$}
			    \State Update $\alpha_{(p(j), j)}$ according to \eqref{eq:alpha_updates}
			\EndFor
		\EndWhile
		\State \textbf{Output:} Optimal dual potentials $(u_1,\dotsc,u_N)$ and vectors $(\alpha_{(j,k)})_{(j,k) \in \Bar{\E}}$
		\caption{Sinkhorn Iterations for Tree-Structured Costs \eqref{eq:mult_bary}}
		\label{alg:sinkhorn_for_trees}
	\end{algorithmic}
\end{algorithm*}

Consider a tree $\G = (\V,\E)$ as in Subsection~\ref{subsec:tree-costs} and corresponding cost functions
\begin{equation}\label{eq:cost_decoupling}
    c = \bigoplus_{(j,k) \in \E} c^{(j,k)}, \quad c^{(j,k)} \in \R_{\geq 0}^{M_j \times M_k}.
\end{equation}
Then, it holds $K_{i_1, \dots, i_N} = \prod_{(j, k)\in \E}K^{(j, k)}_{i_j, i_k}$, where
\[
K^{(j,k)} \coloneqq \exp(-c^{(j,k)}/\eps) \in \R_+^{M_j \times M_k}.\]
The next result, c.f.\ \cite[Thm.~3.2]{HRCK20}, is the main ingredient for an efficient implementation of Algorithm~\ref{def:sinkhorn} for solving $\mathrm{UMOT}_{\eps}$ with tree-structured cost functions.

\begin{theorem}\label{thm:marginals_trees}
The projection onto the $j$-th marginal of $K \odot u$ is given by
\[
{P_{\X_j}}\strut{}_\#(K \odot u) = u_j \odot \bigodot_{l \in \mathcal{N}_j} \alpha_{(j,l)}.
\]
Here the $\alpha_{(j,k)}$ are computed recursively for $(j,k) \in \bar \E \coloneqq \{(v,w)| (v,w) \in \mathcal E \text{ or } (w,v) \in \mathcal E \}$ starting from the leaves by
\begin{equation}\label{eq:alpha_updates}
    \alpha_{(j,k)} = K^{(j,k)} \biggl(u_k \odot \bigodot_{l \in \mathcal N_k \setminus \{j\}} \alpha_{(k,l)}\biggr),
\end{equation}
with the usual convention that the empty product is $1$.
\end{theorem}

First, we traverse the tree $\G$ by a pre-order depth-first search with respect to a root $v_0$.
This results in a strict ordering of the nodes, which is encoded in the list $\V_{\rightarrow}$.
Every node $k \in V$ except the root has a unique parent, denoted by $p(k)$.
We  denote by $\V_{\leftarrow}$ the reversed list $\V_{\rightarrow}$.
Note that the order in which we update the vectors $(\alpha_{(j,k)})_{(j,k) \in \Bar{\E}}$ and potentials $(u_j)_{j \in \mathcal V}$ in Algorithm~\ref{alg:sinkhorn_for_trees} fits to the underlying recursion in \eqref{eq:alpha_updates}.
Furthermore, the computational complexity of Algorithm~\ref{alg:sinkhorn_for_trees} is linear in $N$.
More precisely, $2(N-1)$ matrix-vector multiplications are performed to update every $u_j$ once, which is in alignment with the two-marginal case.
In particular, solving $N$ decoupled problems has the same complexity per iteration with the disadvantage that the marginals of the obtained transport plans do not necessarily fit to each other.
Although we discussed Algorithm~\ref{alg:sinkhorn_for_trees} mainly for computing barycenters, it can also be applied without free marginals, see Section~\ref{sec:tracking_transfer} for a numerical example.

\section{Numerical Examples}\label{sec:numerical_examples}
In this section, we present three numerical examples, 
where the first two confirm our theoretical findings from Section~\ref{sec:BaryInt}. 
Part of our Python implementation is built upon the POT toolbox \cite{POT-toolbox}.

From now on, all measures are of the form 
$\sum_{k=1}^m \mu^k \delta_{x_k}$ with support points $x_k\in \mathbb R^d$, $k=1,\ldots,m$.
We always use
the cost functions $c(x, y) = \norm{x-y}^2$ with corresponding cost matrices $c = c(x_j,x_k)_{j,k=1}^m$.
All reference measures are the counting measure, see Remark~\ref{rem:kl}iii), 
i.e., we exclusively deal with entropy regularization. 

\subsection{Barycenters of 1D Gaussians}
We start with computing the barycenter for two simple measures $\mu_1$, $\mu_2$, which are produced by sampling truncated normal distributions 
$\mathcal N(0.2,0.05)$ and $2\mathcal N(0.8,0.08)$ on $[0,1]$ on a uniform grid. 
Clearly, this choice makes an unbalanced approach necessary.
As before, we denote the discrete spaces and those of the
barycenter by $\X_1 = \X_2 = \Y$.
To approximate the marginals, we use the Shannon entropy functions $\varphi = \varphi_1 = \varphi_2$ so that $D_\varphi = \KL$.
First, we solve the barycenter problem \eqref{eq:ub_bary}, which reads for $t_1 = 1 - t$ and $t_2 = t \in (0,1)$ as
\begin{align}\label{2}
     \hat \xi = &\argmin_{\xi}
		 \sum_{i=1}^2 t_i \min_{ \pi^{(i)} }  \langle c, \pi^{(i)} \rangle 
		+ \eps  E(\pi^{(i)} )\\
		&+ t_i\KL \bigl( P_{\X_i}\strut{}_\# \pi^{(i)} , \mu_i  \bigr) \\
		&
		\mathrm{subject \; to} \quad {P_{\Y}}\strut{}_\# \pi^{(1)} = {P_{\Y}}\strut{}_\# \pi^{(2)} = \xi.
    \end{align}
The resulting barycenter for $\eps=0.005$ is computed using the Sinkhorn iterations described in \cite[Sec.~5.2]{CPSV17}
and is shown in Fig.~\ref{fig:umot_vs_uot} for different $t \in (0,1)$ together with the marginals
$\tilde \mu_i \coloneqq P_{\X_i}\strut{}_\# \hat \pi^{(i)}$, $i=1,2$.
 
We compare these barycenters with the marginal $P_{\Y}\strut{}_\# \hat \pi$ of the corresponding optimal plan $\hat \pi$ for the multi-marginal problem
\begin{align}\label{mult}
     \argmin_{\pi \in \mathcal M^+(\X \times \Y)}&
		\langle c_{t}, \pi \rangle + \eps  E(\pi)\\ &+ 
		 \sum_{i=1}^2 t_i\KL \bigl( P_{\X_i}\strut{}_\# \pi , \mu_i\bigr) 
 \end{align}
computed by Algorithm~\ref{alg:sinkhorn_for_trees}.
The resulting marginals are provided in Fig.~\ref{fig:umot_vs_uot}.
As explained in Remark \ref{rem:smear}, the barycenters $\hat \xi$ appear smoothed compared to $P_{\Y}\strut{}_\# \hat \pi$.
As already mentioned, we do not have relations with shortest paths due to the missing metric structure.

\renewcommand\curfolder{img}
\renewcommand\curwidth{0.9\textwidth}
\begin{figure*}[p!]
	\centering
	\includegraphics[width=\curwidth]{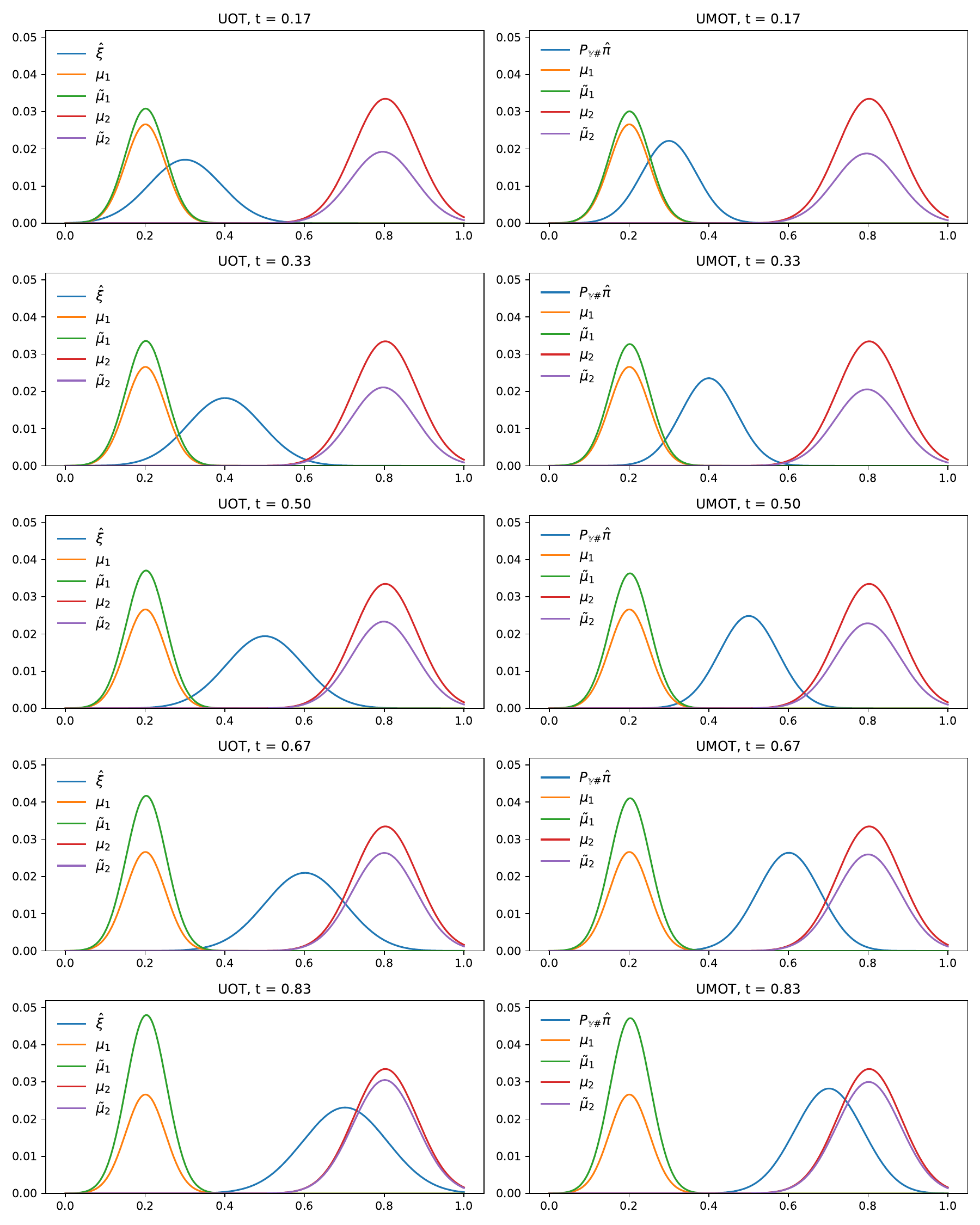}
	\caption{Comparison between the unbalanced barycenter $\hat \xi$ of two Gaussians (left) 
	and the marginal $P_{\Y}\strut{}_\# \hat \pi$ of the optimal transport plan for the corresponding	multi-marginal 
	problems (right) with different $t \in (0,1)$.
	The input measures $\mu_i$, $i=1,2$ and 
	the corresponding marginals $\tilde \mu_i$, $i=1,2$, of the respective transport plans are shown as well.}
	\label{fig:umot_vs_uot}
\end{figure*}

\subsection{H-tree Shaped Cost Functions}\label{sec:h-tree}

Next, we turn our attention to the ``interpolation'' of four gray-value images of size $100 \times 100$
considered as probability measures $\mu_v$, $v = 1,3,5,7$,  
along a tree that is H-shaped, see Fig.~\ref{fig:trees}.
The images are depicted in the four corners of Fig.~\ref{fig:Htree_MOT}.
In this example, the measures corresponding to the inner nodes with $u = 2,4,6$ have to be computed.
For this purpose, we choose $\eps=4\cdot 10^{-4}$ and $D_{\varphi_v}(\cdot, \mu_v)=0.05\KL(\cdot, \mu_v)$. 

\paragraph{Comparison with $\MOT_\eps$}
As the measures have the same mass, we can
compare our proposed $\UMOT_\eps$ approach 
with the balanced $\MOT_\eps$ method.
Equal costs $c$ as well as equal weights are assigned to the edges.
Note that the total cost decouples according to the H-shaped tree.
The obtained results for $\MOT_\eps$ and $\UMOT_\eps$ are depicted 
in Figs.~\ref{fig:Htree_MOT} and \ref{fig:Htree_UMOT}, respectively.
For $\UMOT_\varepsilon$, the corners contain the marginals 
$P_v \hat \pi = P_{\X_v}\strut{}_\# \hat \pi$ instead of the given measures.
As the mass in the different image regions is different, mass is transported between them in the $\MOT_\eps$ interpolation. 
In contrast, only a minimal amount of mass is transported between the images regions for $\UMOT_\eps$, where the mass difference is compensated by only approximately matching the prescribed marginals.
This behavior can be controlled by adjusting the weights in the $\varphi$-divergences.

While $\MOT$ becomes numerically unstable for smaller $\varepsilon$ than $4\cdot 10^{-4}$, the
$\UMOT$ problem remains solvable for smaller values of  $\varepsilon$.
In our numerical experiments, which are not reported here, this led to less blurred images.

\renewcommand\curwidth{0.48\textwidth}
\begin{figure*}[t!]
	\centering
		\subfloat[$\MOT_\eps$: Given measures $\mu_v$, $v=1,3,5,7$ and marginals $P_{\X_u}\strut{}_\#  \hat \pi$, $u=2,4,6$.
	\label{fig:Htree_MOT}]
	{\includegraphics[width=\curwidth]{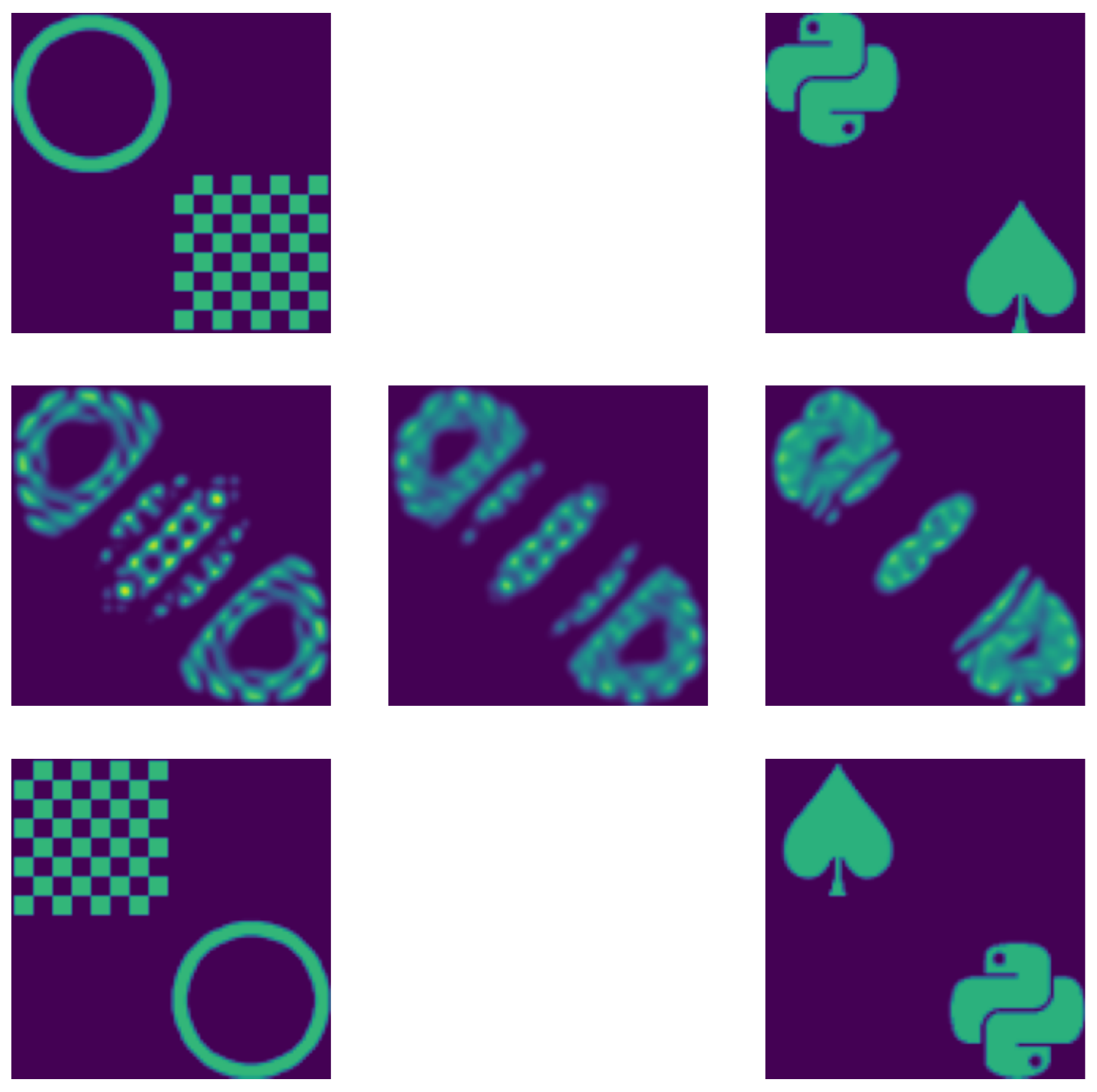}} 
	\hspace{.3cm}
	\subfloat[$\UMOT_\eps$: Marginals $P_{\X_u}\strut{}_\# \hat \pi$, $u=1,\ldots,7$.
	\label{fig:Htree_UMOT}]
	{\includegraphics[width=\curwidth]{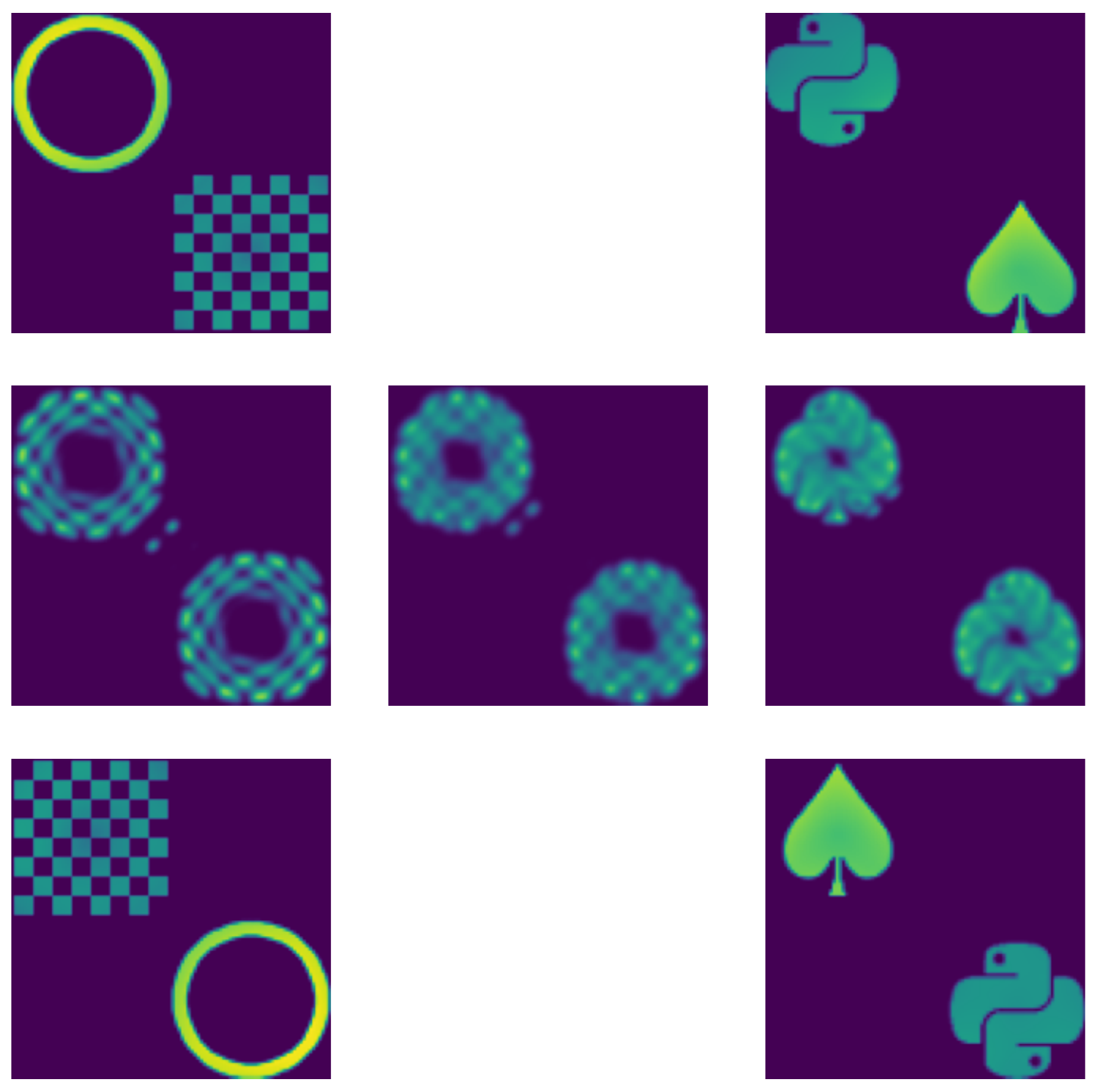}}
	\\[.3cm]
	\subfloat[$\UMOT_\eps$ marginals, multiple star-shaped.\label{fig:Htree_UMOT-star}]{\includegraphics[width=\curwidth]{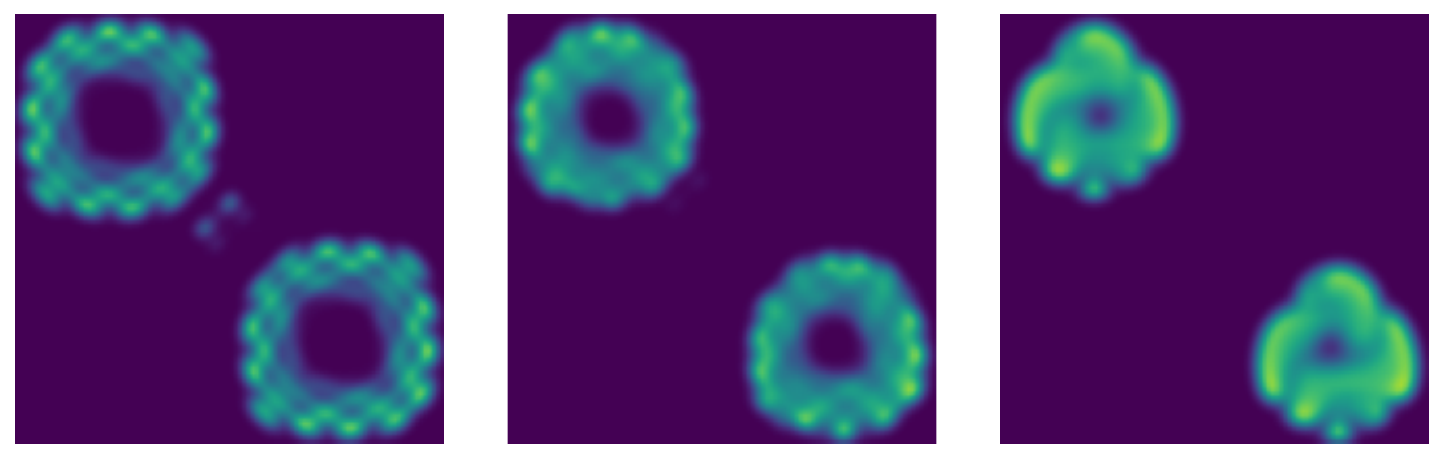}}
	\hspace{.3cm}
	\subfloat[$\UOT_\eps$ barycenters, multiple star-shaped.
	\label{fig:Htree_sumUOT}]
	{\includegraphics[width=\curwidth]{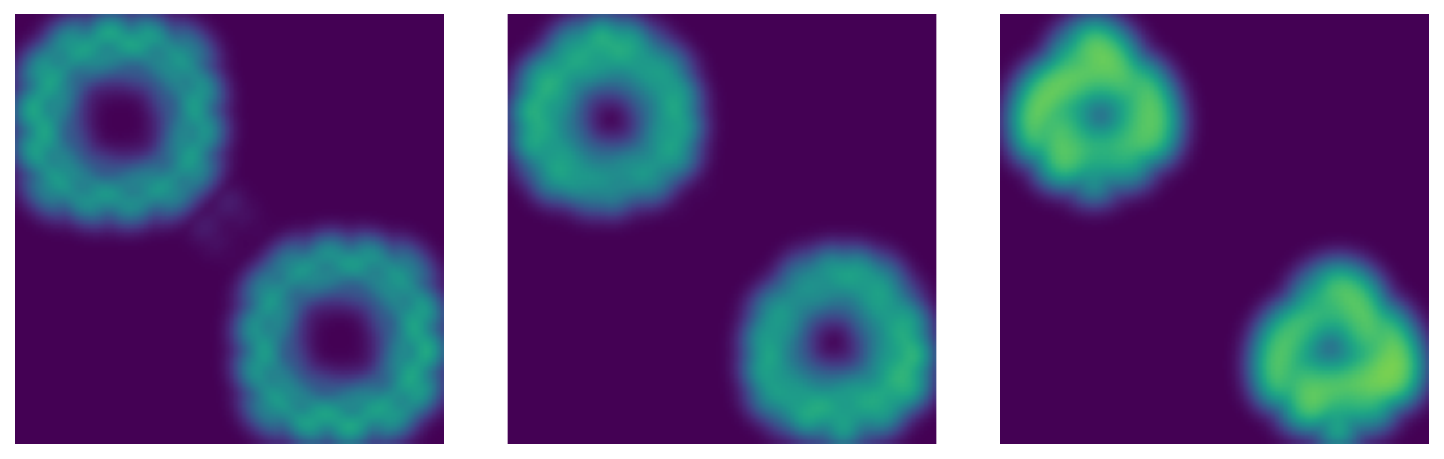}}
	\caption{Different ``interpolations'' between the four given images 
	at the corners of Fig.~\ref{fig:Htree_MOT}. Top:  H-tree structured cost function.
	Bottom: Star-shaped decomposition approach.
	A common color coding is used for all images.}
	\label{fig:Htree}
\end{figure*}

\paragraph{Comparison with multiple star graph barycenters}
Next, we provide a heuristic comparison with an alternative interpolation approach.
Instead of computing three interpolating measures simultaneously, we solve three indidivual barycenter problems based on star-shaped graphs with leaves corresponding to $\mu_v$, $v = 1,3,5,7$.
This is done both with the $\UOT_\eps$ and corresponding $\UMOT_\eps$ approach.
More precisely, we solve \eqref{eq:ub_bary} and \eqref{eq:xxl} three times with weights 
$1/12\cdot(5, 5, 1, 1)$, $1/4\cdot(1, 1, 1, 1)$ and $1/12\cdot(1, 1, 5, 5)$,
respectively.
These weights have been derived from the solution of the balanced H-graph-Problem with Dirac measures at the leaves, which is easy to obtain from a linear system corresponding to the first order optimality conditions.
The results are provided in Figs.~\ref{fig:Htree_UMOT-star} and \ref{fig:Htree_sumUOT}.
Noteworthy, both interpolations have 
an even less pronounced mass transfer between the different image structures.
However, the computed images look considerably smoother than their counterparts in Fig.~\ref{fig:Htree_UMOT}.
Again, as expected, we observe that the $\UOT_\eps$ results in Fig.~\ref{fig:Htree_sumUOT} 
are more blurred than the corresponding $\UMOT_\eps$ interpolations in Fig.~\ref{fig:Htree_UMOT-star}.
As they are all marginals of multi-marginal transport plans, the images in Figs.~\ref{fig:Htree_MOT} and \ref{fig:Htree_UMOT} have the same mass.
In contrast, the images in Figs.~\ref{fig:Htree_UMOT-star} and \ref{fig:Htree_sumUOT} 
do not necessarily have the same mass as they are marginals of different transport plans.
Hence, depending on the application, one or the other approach might be preferable.

Note that in order to compute the multiplications with the matrices $K$ in Algorithm~\ref{alg:sinkhorn_for_trees}, we exploit the fact that the Gaussian kernel convolution $K$ is separable in the two spatial dimensions 
and can be performed over the rows and columns of the images one after another, such that we never actually store the matrix $K\in \R^{10000\times 10000}$.
Consequently, we cannot use stabilizing absorption iterations as proposed in \cite{CPSV17}.

\subsection{Particle Tracking and Transfer Operators}\label{sec:tracking_transfer}

\renewcommand\curfolder{img}
\renewcommand\curwidth{0.9\textwidth}
\begin{figure*}[t!]
	\centering
	\includegraphics[width=\curwidth]{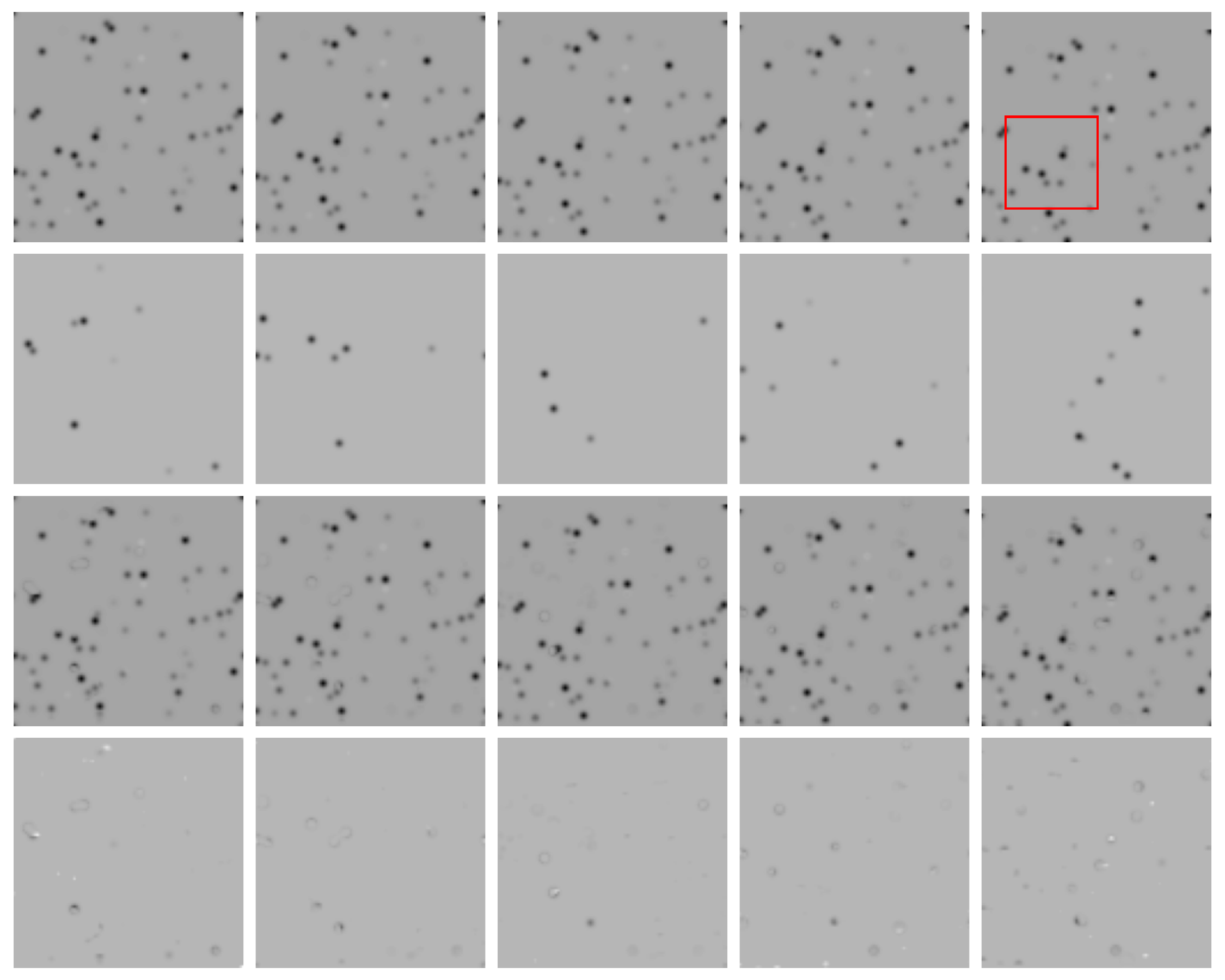}
	\caption{
	Top: Clean data set. 
	Middle upper: Added noise to get the measures $\mu_i$, $i=1,\ldots,5$. 
	Middle lower: Marginals $P_{\X_i}\strut{}_\# \hat \pi$, $i=1,\ldots,5$, of optimal transport plan $\hat \pi$ for $\UMOT_\eps(\mu)$.
	Bottom: Remaining noise and artifacts (marginals -- clean).}
	\label{fig:dots}
 \includegraphics[width=\curwidth]{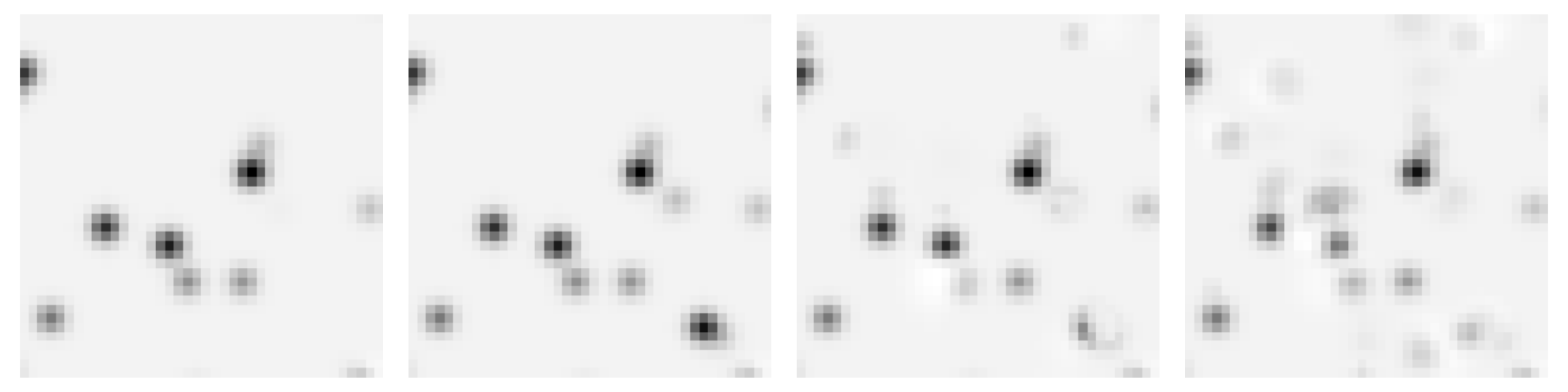}
	\caption{Zoom into marked patch in Fig.~\ref{fig:dots}. Note that the brightness scale differs from Fig.~\ref{fig:dots} for better contrast.
	Left: Ground truth $\mu_{ \mathrm{clean},5 }$. 
	Middle left: Noisy $\mu_{ 5 }$.
	Middle right: Transferred image $K_{\UMOT}^\tT \, \mu_{\mathrm{clean},1}$. 
	Right: Transferred image  $K_{\UOT}^\tT  \, \mu_{\mathrm{clean},1}$.
	}
	\label{fig:dots_propagated}
\end{figure*}

Finally, we investigate whether computing a single joint $\UMOT_\eps$ solution 
can be beneficial compared to computing several $\UOT_\eps$ plans sequentially, e.g., for particle tracking. 
To this end, we create a time-series of five images. Each image has size $100\times 100$ pixels and consists of ``dots'' by sampling uniform noise, 
setting all values above a small threshold to zero, and applying a Gaussian filter with small variance. 
One time step corresponds to shifting the image by two pixels downwards filling with the small constant background value from the top, which results in images $\mu_{\mathrm{clean},i}$, $i=1, \dots,5$.
We modify this time-series of five images by adding dots randomly for every time step in a similar manner.
Consequently, the data consists of a drift component 
and some random dots popping up and disappearing again.
The resulting data $\mu_i$, $i=1, \dots,5$, is shown in Fig.~\ref{fig:dots}.
As we want to extract only the drift component, we apply $\UMOT_\varepsilon(\mu)$ for
a line-tree-structured cost function with the same costs $c_i$ along the path,
regularization parameter $\eps = 10^{-4}$,  
and 
$D_{\varphi_i}(\cdot,\mu_i)=7\cdot 10^{-4} \TV(\cdot, \mu_i)$, $i=1,\ldots,5$.
We expect that the hard thresholding of the corresponding $\aprox$-operators for 
$\TV(\cdot, \mu_i)$, $i=1,\ldots,5$, is particularly well suited for removing the noise in our example, 
see also \cite{CM10,F10}. 
The resulting marginals of the optimal transport plan $P_{\X_i}\strut{}_\# \hat \pi$ are shown in Fig.~\ref{fig:dots}.
Indeed, the method manages to remove most of the noise dots.

\paragraph{Transfer operators.}
For our next comparison, we use OT-related \emph{transfer operators}, which
have recently been discussed in \cite{KLNS20}.
We abuse notation in this section by sometimes identifying empirical measures with their vectors or matrices of weights for simplicity.
In a nutshell, assuming a discrete, two-marginal setting of probability measures
with optimal transport matrix $\hat \pi$ 
and left marginal vector
$\mu_\ell = (\mu_\ell^k)_{k=1}^m$, 
we can define a row-stochastic transition matrix $K$ by setting 
$$
K \coloneqq \diag\left(\mu_\ell^{-1}\right) \hat \pi.
$$ 
This concept allows us to propagate other measures than $\mu_\ell$ forward in time by multiplication with 
$K^\tT$.
Note that there is a continuous analog in terms of \emph{Frobenius--Perron-operators}, \emph{Markov kernels} and the \emph{disintegration theorem}, see \cite{brin2002introduction, boyarsky1997laws, LasotaAndrzej1994CFaN, KNKW18, froyland2013analytic} for details. 

Now, we compute the marginal $\hat \pi_{1,5} \coloneqq P_{\X_1 \times \X_5}\strut{}_\# \hat \pi$ of the optimal $\UMOT_\eps(\mu)$ transport plan $\hat \pi$. 
Using the marginal 
$\tilde \mu_1 \coloneqq P_{\X_1}\strut{}_\# \hat \pi_{1,5} = P_{\X_1}\strut{}_\# \hat \pi$, we get
the transfer operator
$$
K_{\UMOT} = \diag \left(\tilde \mu_1^{-1}\right) \hat \pi_{1,5}.
$$
Then, we propagate the first clean image $\mu_{\mathrm{clean},1}$ by this transfer operator, i.e., we compute
$K_{\UMOT}^\tT \, \mu_{\mathrm{clean},1}$.
The result is shown in  Fig.~\ref{fig:dots_propagated}.

For comparison, we also compute $N-1$ successive 
$\UOT_\eps(\mu_i,\mu_{i+1})$ plans $\hat \pi^{(i)}$, $i=1, \dots, 4$. 
Denoting the marginals by $\tilde \mu_i \coloneqq P_{\X_i}\strut{}_\# \hat \pi^{(i)}$, $i=1, \dots, 4$, we consider the transfer kernel
\begin{equation}
    K_{\UOT} = \prod_{i=1}^{4}  \diag \left(\tilde \mu_i^{-1}\right) \hat \pi^{(i)}.
\end{equation}
Then, we transfer the clean first image by this operator, i.e., we compute $K_{\UOT}^\tT \, \mu_{\mathrm{clean},1}$.
The obtained results are shown in Fig.~\ref{fig:dots_propagated}.
Note that similarly as described in the previous subsection, 
the computations can be carried out using separable convolutions without ever storing the large matrix $K$.

As we wanted to extract the drift behavior using only the noisy images, the propagated images should be compared to $\mu_{\mathrm{clean},5}$, i.e., the last image of the first row in Fig.~\ref{fig:dots}, which is just a shifted version of the first image.
In some sense, we can interpret this image as the propagation using the ``true'' transfer operator.
There are considerably less artifacts visible in the $\UMOT_\eps$  propagated image  compared to the $\UOT_\eps$ version.
This is particularly pronounced in the middle left part of the images.
As an error measure, we also computed the squared Euclidean distances between the propagated images and the ground truth, which are $2.98$ and $6.48$ for the $\UMOT_\eps$ and $\UOT_\eps$ version, respectively.

From an intuitive point of view, the results are not surprising.
If we are only provided with a single pair of images from the sequence, it is much harder to judge which points correspond to noise than for a whole sequence of images.
Note that a single Sinkhorn iteration for the coupled $\UMOT_\eps$ problem has the same computational complexity as for all of the decoupled $\UOT_\eps$ problems together.
Hence, the $\UMOT_\eps$ approach appears to be better suited for this application, as it incorporates more information about the problem without significant additional computational cost.

\section{Conclusions}\label{sec:conc}
In this paper, we introduced a regularized unbalanced multi-marginal optimal transport framework, abbreviated 
$\UMOT_\eps$, bridging the gap between regularized unbalanced optimal transport
and regularized multi-marginal optimal transport.
We outlined how the Sinkhorn algorithm can be adapted to solve $\UMOT_\eps$ efficiently for tree-structured costs.
For this case, we have also shown how $\UMOT_\eps$ provides alternative solutions of barycenter-like problems with desirable properties, such as improved sharpness.
In the future, we plan to examine $\UMOT_\eps$ in connection with particle cluster tracking methods, e.g.,
following the ideas in \cite{KLNS20}.
Furthermore, it would be interesting to examine a reformulation of the regularized unbalanced barycenter problem as UMOT problem using a cost function similar to that 
in Remark \ref {rem:bary-mot}.
Additionally, we want to investigate the theoretical relation between the two interpolation approaches from Sec.~\ref{sec:h-tree}.
Finally, we are interested in $\UMOT_\eps$ for measures having the same moments up to a certain order and for measures living on special manifolds such as, e.g., torus or spheres, see also \cite{EGNS2019}.

\section*{Acknowledgment}
The datasets generated during and analyzed during the current study are not publicly available, but are available from the corresponding author on reasonable request.
Funding by the DFG under Germany's Excellence Strategy – The Berlin Mathematics Research Center MATH+ (EXC-2046/1, Projektnummer: 390685689)
and by the DFG Research Training Group DAEDALUS (RTG 2433)
is acknowledged.

\bibliographystyle{abbrv}
\bibliography{references_clean}
\end{document}